\documentclass[letterpaper, 12pt]{article}
\usepackage{amssymb,amsmath,amsthm}
\usepackage{epsf}
\usepackage{times,bm,mathrsfs}
\usepackage{amsmath}
\usepackage{amssymb}
\usepackage{amsfonts}
\usepackage{mathrsfs}
\usepackage{bbm}
\usepackage{color}
\usepackage{graphicx}
\usepackage{amscd}
\usepackage{epsfig}
\usepackage{comment}
\usepackage{tikz, tikz-cd}
\usetikzlibrary{fit, positioning, arrows}

\definecolor{Red}{rgb}{1,0,0}
\definecolor{Blue}{rgb}{0,0,1}

\newtheorem{theorem}{Theorem}
\newtheorem{definition}{Definition}

\newtheorem{lemma}{Lemma}

\newtheorem{prop}{Proposition}

\newtheorem{example}{Example} 
\newtheorem{question}{Question} 

\def\beq{ \begin{equation} }
\def\eeq{ \end{equation} }

\def\square{\vcenter{\vbox{\hrule height .4pt
  \hbox{\vrule width .4pt height 5pt \kern 5pt
        \vrule width .4pt} \hrule height .4pt}}}

\def\var{\hbox{var}\,}

\def\P{{\mathbb P}}     
\def\E{{\mathbb E}}     
\def\<{{\langle}} 
\def\>{{\rangle}} 

\newcommand{\spread}{\mathrm{Spr}}
\renewcommand{\var}{\mathrm{Var}}
\newcommand{\cov}{\mathrm{Cov}}
\newcommand{\tv}{\mathrm{TV}}
\newcommand{\path}{\mathrm{P}}

\begin{document}
\title{Necessary and sufficient conditions for consistent
	root reconstruction in Markov models on trees
}
\author{Wai-Tong (Louis) Fan\footnote{Department of Mathematics, 
		UW--Madison.
		Work supported by NSF grant DMS--1149312 (CAREER) to SR.}
	\and
	Sebastien Roch\footnote{Departments of Mathematics, 
		UW--Madison.
		Work supported by NSF grants DMS-1149312 (CAREER) and DMS-1614242.}}

\date{\today}
\maketitle

\begin{abstract}
We establish necessary and sufficient conditions for consistent root reconstruction in conti\-nuous-time Markov models with countable state space on bounded-height trees. 
Here a root state estimator is said to be consistent if the probability that it returns to the true root state converges to $1$ as the number of leaves tends to infinity. 
We also derive quantitative bounds on the error of reconstruction.
Our results answer a question of Gascuel and Steel \cite{GascuelSteel:10} and have implications for ancestral sequence reconstruction in a classical evolutionary model of nucleotide insertion and deletion \cite{thorne1991evolutionary}.
\end{abstract}

\thispagestyle{empty}

\clearpage


\section{Introduction}




\paragraph{Background}
In biology, the inferred evolutionary history of organisms and their relationships is depicted diagrammatically as a phylogenetic tree, that is, a rooted tree whose leaves represent living species and branchings indicate past speciation events~\cite{Felsenstein:04}. The evolution
of species features, such as protein sequences, linear arrangements of genes on a chromosome or the number of horns of a lizard, is commonly assumed to follow Markovian dynamics along this tree~\cite{Steel:16}. That is, on each edge of the tree, the state of the feature changes according to a continuous-time Markov process; at bifurcations, two independent copies of the feature evolve along the outgoing edges starting from the state at the branching point. The length of an edge is a measure of the expected amount of change along it. See Section~\ref{sec:defs} for a formal definition.

In this paper, we are concerned with the problem of inferring an ancestral state from observations at the leaves of a given tree under known Markovian dynamics. We refer to this problem, which has important applications in biology~\cite{thornton2004resurrecting,liberles2007ancestral}, as the {\bf root reconstruction problem}. Many rigorous results have been obtained in the finite state space case, although much remains to be understood; see, e.g.,~\cite{KestenStigum:66,BlRuZa:95,Ioffe:96a,EvKePeSc:00,Mossel:01,MosselPeres:03,BoChMoRo:06,sly2009reconstruction,bhatnagar2010reconstruction,bhatnagar2011reconstruction,Sly:11} for a partial list. Typically, one seeks an estimator of the root state which is strictly superior to random guessing---uniformly in the depth of the tree---under a uniform prior on the root~\cite{Mossel:01}. Whether such an estimator exists has been shown to hinge on a trade-off between the mixing rate of the Markov process (i.e., the speed at which information is lost) and the growth rate of the sequence of trees considered (i.e., the speed at which information is duplicated). In some cases, for instance two-state symmetric Markov chains on $d$-ary trees~\cite{KestenStigum:66,Ioffe:96a}, sharp thresholds have been established. 

\paragraph{Main results}
Here, we study the root reconstruction problem in an alternative setting where estimators with stronger properties can be derived. We consider sequences of nested trees with {\it uniformly bounded depths}. This is motivated by contemporary applications in evolutionary biology where the rapidly increasing availability of data from ever-growing numbers of organisms, particularly genome sequencing data, has allowed dense sampling of species within the same family or genus. This is sometimes referred to as the {\bf taxon-rich setting} and has been considered in a number of recent theoretical studies~\cite{GascuelSteel:10,ho2013,FanRoch:u}. As shown in~\cite{GascuelSteel:10}, a key difference with the traditional setting described above
is that, in the taxon-rich setting, {\bf consistent root state estimation} is possible. 
In this context, a consistent estimator is one whose probability of success tends to $1$ as the number of leaves goes to infinity.  See Section~\ref{sec:defs} for a formal definition.
In particular, for general finite-state-space Markov processes on ultrametric trees, i.e., trees whose leaves are equidistant from the root, Gascuel and Steel~\cite{GascuelSteel:10} give sufficient conditions for the existence of consistent root state estimators by introducing a notion of ``well-spread trees.'' 

Building on this work, we give both necessary and sufficient conditions for consistent root reconstruction for general trees and general Markov processes on countable state spaces, a question left open in~\cite{GascuelSteel:10}.
On an intuitive level, the greater the number of leaves, the more information we have about the root state. However, the leaves do not provide {\it independent} information due to the correlation arising from the partial overlap of the paths from the root to the leaves. 
In particular we cannot appeal, for instance, to the consistency of maximum likelihood estimation for independent samples~\cite{LehmannRomano:05}. We show however that, under a certain ``root density'' assumption we refer to as the big bang condition, one can identify a subset of leaves that are ``sufficiently independent.'' We also derive quantitative bounds on the error of reconstruction in terms of natural properties of the tree sequence and Markov process.






One applied motivation for our results, especially our consideration of countable state spaces, is ancestral sequence reconstruction in DNA evolution models accounting for nucleotide insertion and deletion. Our main theorem immediately gives necessary and sufficient conditions for the existence of consistent root estimators for a classical such model known as the TKF91 process~\cite{thorne1991evolutionary}. This is detailed in Section~\ref{sec:tkf}.
In this context, our work is also related to trace reconstruction, which corresponds roughly to the star tree case under simplified analogues to the TKF91 process. See, e.g.,~\cite{mitzenmacher2009survey} for a survey.
See also~\cite{andoni2012global} for related work in the phylogenetic setting.

\paragraph{Organization} Definitions and main results are stated in Sections~\ref{sec:defs} and~\ref{sec:results}. The connection between our big bang condition and the well-spread trees of~\cite{GascuelSteel:10} is established in Section~\ref{sec:spread}. Our impossibility result is proved in Section~\ref{sec:impossibility}, while our consistency result and error bound are detailed respectively in Sections~\ref{sec:consistency} and~\ref{section:rate}.

\subsection{Basic definitions}
\label{sec:defs}

\paragraph{Markov chains on trees}
We consider the following class of latent tree models arising in phylogenetics.
The model has two main components:
\begin{itemize}
\item The first component is a tree. More precisely, throughout, by a tree we mean a finite, edge-weighted, rooted tree $T = (V,E,\rho,\ell)$, where $V$ is the set of vertices, $E$ is the set of edges oriented away from the root $\rho$, and $\ell:E \to (0,+\infty)$ is a positive edge-weighting function. 
We denote by $\partial T$ the leaf set of $T$. 
No assumption is made on the degree of the vertices. We think of $T$ as a continuous object, where each edge $e$ is a line segment of length $\ell_{e}$ and whose elements we refer to as points. We let $\Gamma_T$ be the set of points of $T$.

\item The second component is a time-homogeneous, continuous-time Markov process taking values in a countable state space $\mathcal{S}$.
Without loss of generality, we let $\mathcal{S} = \{1,\ldots,|\mathcal{S}|\}$ in the finite case and $\mathcal{S} = \{1,2,\ldots\}$ in the infinite case.
	We denote by 
	$\mathbf{P}_t=(p_{ij}(t):\,i,j\in\mathcal{S})$ the transition matrix at time $t\in [0,\infty)$, 
that is, $p_{ij}(t)$ is the probability that the
	state at time $t$ is $j$ given that it was $i$ at time $0$. We also let
	\begin{equation}
	\label{eq:def-row}
	\mathbf{p}^{i}(t)=(p_{i1}(t),\,p_{i2}(t),\,\ldots),
	\end{equation} 
	be the $i$-th row in the transition matrix.
	We assume that $(\mathbf{P}_t)_t$ admits a $Q$-matrix $Q=(q_{ij}:\,i,j\in\mathcal{S})$ which is stable and conservative, that is, 
\begin{equation*}
q_{ij} :=    \left.\frac{d}{dt}p_{ij}(t)\right|_{t=0} \in[0,\infty) \qquad \forall i\neq 	j,
    \end{equation*}
and
\begin{equation}
	\label{eq:def-qi}
	q_i:=-q_{ii}=\sum_{j\neq i}q_{ij}\in [0,\infty), \qquad \forall i.
	\end{equation}
	See, e.g.,~\cite[Chapter 2]{Liggett:10} or \cite{Anderson:91} for more background on continuous-time Markov chains.
 \end{itemize}
We consider the following stochastic process indexed by the points of $T$.
The root is assigned a state $X_{\rho}\in \mathcal{S}$, which is drawn from a probability distribution on $\mathcal{S}$. This state is then propagated down the tree according to the following recursive process. Moving away from the root, along each edge $e = (u,v) \in E$,
conditionally on the state $X_u$, we run the Markov process $\mathbf{P}_t$ started at $X_u$ for an amount of time $\ell_{(u,v)}$. We denote by $X_\gamma$ the resulting state at $\gamma \in e$. We call the process $\mathcal{X} = (X_\gamma)_{\gamma \in \Gamma_T}$ a \textbf{$\mathbf{P}_t$-chain on $T$}. For $i\in\mathcal{S}$, we let $\P^i$ be the probability law when the root state $X_{\rho}$ is $i$. If $X_{\rho}$ is chosen according to a distribution $\pi$, then we denote the probability law by $\P^{\pi}$. Note that the leaf distribution conditioned on the root state is given by
\begin{align}
\mathcal{L}^i_T\left((x_u)_{u \in \partial T}\right)
:=
\P^i\left[(X_u)_{u \in \partial T} = (x_u)_{u \in \partial T}\right]
=\sum_{\substack{(x'_u)_{u\in V}\,: \\ (x'_u)_{u \in \partial T} = (x_u)_{u \in \partial T}, \\ x'_\rho = i}}\prod_{e = (u,v)\in E} p_{x'_u,x'_v}(\ell_{e}),
\label{eq:def-lit}
\end{align}
for all $(x_u)_{u \in V} \in \mathcal{S}^{\partial T}$.

	
\paragraph{Root reconstruction}
In the {\bf root reconstruction problem} we seek a good estimator of the root state $X_\rho$ based on the leaf states $X_{\partial T}$. More formally, let $\{T^k = (V^k, E^k, \rho^k, \ell^k)\}_{k \geq 1}$ be a sequence of trees with $|\partial T^k| \to +\infty$ and let $\mathcal{X}^k = (X^k_\gamma)_{\gamma \in \Gamma_{T^k}}$ be a $\mathbf{P}_t$-chain on $T^k$ with root state distribution $\pi$. 
\begin{definition}[Consistent root reconstruction]
A sequence of root estimators
$$
F_k:\mathcal{S}^{\partial T^k} \to \mathcal{S}, 
$$
is said to be {\em consistent} for $\{T^k\}_k$, $(\mathbf{P}_t)_t$ and $\pi$ if
$$
\liminf_{k \to +\infty}
\P^\pi\left[
F_k\left(X^k_{\partial T^k}\right) = X^k_{\rho^k}
\right] = 1.
$$
\end{definition}
\noindent The basic question we address is the following.
\begin{question}
Under what conditions on  $\{T^k\}_k$, $(\mathbf{P}_t)_t$, and $\pi$ does there exist a sequence of consistent root  estimators?
\end{question}
\noindent Before stating our main theorems, we make some assumptions and introduce further notation.

\paragraph{Basic setup}
For concreteness, we let
$\{T^k\}_{k}$ be a nested sequence of trees with common root $\rho$. That is, for all $k > 1$, $T^{k-1}$ is a restriction of $T^{k}$, as defined next. 
\begin{definition}[Restriction]
\label{def:restriction}
Let $T = (V,E,\rho,\ell)$ be a tree.
For a subset of leaves $L \subset \partial T$,
the {\em restriction of $T$ to $L$} is the 
tree obtained from $T$ by keeping only those
points on a path between the root $\rho$ and 
a leaf $u \in L$.	
\end{definition}
\noindent Observe that a restriction of $T$ is always rooted at $\rho$. Without loss of generality, we assume that $|\partial T^k| = k$, so that $T^k$ is obtained by adding a leaf edge to $T^{k-1}$. (More general sequences can be obtained as subsequences.) In a slight abuse of notation, we denote by $\ell$ the edge-weight function for all $k$.
For $\gamma\in \Gamma_T$, we denote by $\ell_{\gamma}$ the length of the unique path from the root $\rho$ to $\gamma$. We refer to $\ell_\gamma$ as the distance from $\gamma$ to the root. 
Our standing assumptions throughout this paper are as follows.
\begin{enumerate}
	\item[(i)] {\it (Uniformly bounded height)} The sequence of trees $\{T^k\}_{k}$ has uniformly bounded height. Denote by
	$h^{k}:=\max\{\ell_x:\,x\in \partial T^k\}$
	the height of $T^k$. 
	Then the bounded height assumption says that 
	$$
	h^* := \sup_{k}h^{k} < +\infty.
	$$

	\item[(ii)] {\it (Initial-state identifiability)}  The Markov process $(\mathbf{P}_t)_t$ is initial-state identifiable, that is, all rows of the transition matrix $\mathbf{P}_t$ are distinct for all $t\in[0,\infty)$. In other words, given the distribution at time $t$, the initial state of the chain is uniquely determined.
	
\end{enumerate}

\noindent Whether the last assumption holds in general for countable-space, continuous-time Markov processes (that are stable and conservative) seems to be open. 
We show in the appendix that it holds for two broad classes of chains: reversible chains and uniform chains, including finite state spaces. 
(Observe, on the other hand, that in the discrete-time case it is easy to construct a transition matrix which does not satisfy initial-state identifiability.)
 We use the notation $a\land b:=\min\{a,b\}$ and $a\lor b:=\max\{a,b\}$.
For two probability measures $\mu_1$, $\mu_2$ on $\mathcal{S}$, let 
\begin{equation}
\label{eq:def-tv}
\|\mu_1 - \mu_2\|_{\tv}
= \frac{1}{2} \sum_{\sigma \in \mathcal{S}}
\left|
\mu_1(\sigma)
-
\mu_2(\sigma)
\right|
= \sup_{\mathcal{A} \subseteq \mathcal{S}} \left|
\mu_1(\mathcal{A}) 
-
\mu_2(\mathcal{A})
\right|
= 
1 - 
\sum_{\sigma \in \mathcal{S}}
\mu_1(\sigma)
\land
\mu_2(\sigma),
\end{equation}
be the total variation distance between
$\mu_1$ and $\mu_2$. 
(The last equality follows from noticing that
$\|\mu_1 - \mu_2\|_{\tv}
= \frac{1}{2} \sum_{\sigma \in \mathcal{S}}
[\mu_1(\sigma)\lor \mu_2(\sigma)
-
\mu_1(\sigma)\land \mu_2(\sigma)]
$
and
$1
= \frac{1}{2} \sum_{\sigma \in \mathcal{S}}
[\mu_1(\sigma)\lor \mu_2(\sigma)
+
\mu_1(\sigma)\land \mu_2(\sigma)]
$.)
Then initial-state identifiability 
is equivalent to
\begin{equation}
\label{eq:ident-tv}
\|\mathbf{p}^{i}(t) - \mathbf{p}^{j}(t)\|_\tv > 0,
\qquad \forall i\neq j \in \mathcal{S}, t \in (0,\infty), 
\end{equation}
where recall that $\mathbf{p}^{i}(t)$ was defined in~\eqref{eq:def-row}.


\paragraph{Big bang condition}
Our combinatorial condition for consistency
says roughly that the $T^k$s are arbitrarily dense around the root. 
\begin{definition}[Truncation]
For a tree $T = (V,E,\rho,\ell)$, let  
$$
T(s)=\{\gamma\in \Gamma_T:\;\ell_\gamma \leq s\},
$$ 
denote the tree obtained by truncating $T$ at distance $s$ from the root. We refer to $T(s)$ as a {\em truncation} of $T$.
\end{definition}
\noindent See the left-hand side of Figure \ref{Fig:Subtree} for an illustration.
Note that, if $s$ is greater than the height of $T$, then $T(s) = T$.
\begin{definition}[Big bang condition]\label{A:BB}
We say that a sequence of trees $\{T^k\}_k$ satisfies the {\em big bang condition} if:
for all $s\in(0,+\infty)$, we have $|\partial T^k(s)|\to +\infty$ as $k\to+\infty$.
\end{definition}
\noindent See Figure~\ref{Fig:BigB} for an illustration.
\tikzset{
	big dot/.style={
		circle, inner sep=0pt, 
		minimum size=1.2mm, fill=black
	}
}
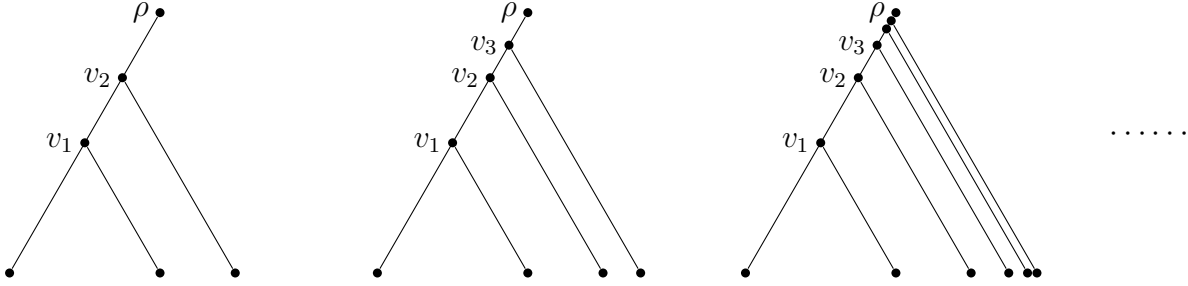
\begin{figure}
	\begin{minipage}{.29\textwidth}
		\begin{tikzpicture}
		\node[big dot] (origin) at (0,0) {};
		\node[big dot] (A) at (-1,-1.732) {};
		\node[big dot] (L1) at (-2,-3.464)   {};
		\node[big dot] (L2) at (0,-3.464)   {};
		\node[big dot] (B) at (-0.5,-0.866)   {};
		\node[big dot] (L3) at (1,-3.464)   {};
		
		\draw[-] (origin) -- (L1);
		\draw[-] (A) -- (L2);
		\draw[-] (B) -- (L3);
		
		\draw (origin) node[left]{$\rho$};
		\draw (A) node[left]{$v_1$};
		\draw (B) node[left]{$v_2$};
		\end{tikzpicture}
	\end{minipage}
	\begin{minipage}{.29\textwidth}
		\begin{tikzpicture}
		\node[big dot] (origin) at (0,0) {};
		\node[big dot] (A) at (-1,-1.732) {};
		\node[big dot] (L1) at (-2,-3.464)   {};
		\node[big dot] (L2) at (0,-3.464)   {};
		\node[big dot] (B) at (-0.5,-0.866)   {};
		\node[big dot] (L3) at (1,-3.464)   {};
		\node[big dot] (C) at (-0.25,-0.433)   {};
		\node[big dot] (L4) at (1.5,-3.464)   {};
		
		\draw[-] (origin) -- (L1);
		\draw[-] (A) -- (L2);
		\draw[-] (B) -- (L3);
		\draw[-] (C) -- (L4);
		
		\draw (origin) node[left]{$\rho$};
		\draw (A) node[left]{$v_1$};
		\draw (B) node[left]{$v_2$};
		\draw (C) node[left]{$v_3$};
		\end{tikzpicture}
	\end{minipage}
	\begin{minipage}{.29\textwidth}
		\begin{tikzpicture}
		\node[big dot] (origin) at (0,0) {};
		\node[big dot] (A) at (-1,-1.732) {};
		\node[big dot] (L1) at (-2,-3.464)   {};
		\node[big dot] (L2) at (0,-3.464)   {};
		\node[big dot] (B) at (-0.5,-0.866)   {};
		\node[big dot] (L3) at (1,-3.464)   {};
		\node[big dot] (C) at (-0.25,-0.433)   {};
		\node[big dot] (L4) at (1.5,-3.464)   {};
		\node[big dot] (D) at (-0.125,-0.2165)   {};
		\node[big dot] (L5) at (1.75,-3.464)   {};
		\node[big dot] (E) at (-0.0625,-0.10825)   {};
		\node[big dot] (L6) at (1.875,-3.464)   {};
		
		\draw[-] (origin) -- (L1);
		\draw[-] (A) -- (L2);
		\draw[-] (B) -- (L3);
		\draw[-] (C) -- (L4);
		\draw[-] (D) -- (L5);
		\draw[-] (E) -- (L6);
		
		\draw (origin) node[left]{$\rho$};
		\draw (A) node[left]{$v_1$};
		\draw (B) node[left]{$v_2$};
		\draw (C) node[left]{$v_3$};
		\end{tikzpicture}
	\end{minipage}
	\begin{minipage}{.10\textwidth}
		$\cdots\cdots$
	\end{minipage}
	\caption{A sequence of trees $\{T^k\}_k$ (from left to right) satisfying the big bang condition. The distance from $v_k$ to the root is $2^{-k}$.}\label{Fig:BigB}
\end{figure}
For $i \in \mathcal{S}$, let $\mathcal{D}_i$ be the set of states reachable from $i$, that is, the states $j$ for which $p_{ij}(t) > 0$ for some $t > 0$ (and, therefore, for all $t > 0$; see e.g.~\cite[Chapter 2]{Liggett:10}).

\subsection{Statements of main results}
\label{sec:results}

Our main result is the following.
\begin{theorem}[Consistent root reconstruction: necessary and sufficient conditions]
	\label{thm:1}
Let $\{T^k\}_k$ and $(\mathbf{P}_t)_t$ satisfy our standing assumptions (i) and (ii), and let $\pi$ be a probability distribution on $\mathcal{S}$. Then there exists a sequence of
root estimators that is consistent for $\{T^k\}_k$, $(\mathbf{P}_t)_t$ and $\pi$ if and only if at least one of the following conditions hold:
\begin{enumerate}
	\item[(a)] {\em (Downstream disjointness)} For all $i \neq j$ such that $\pi(i) \land \pi(j) > 0$, the reachable sets $\mathcal{D}_i$ and $\mathcal{D}_j$ are disjoint.
	
	\item[(b)] {\em (Big bang)} The sequence of trees $\{T^k\}_k$ satisfies the big bang condition.
	
\end{enumerate}
\end{theorem}
\noindent 
An application to DNA evolution by nucleotide insertion and deletion is detailed in Section~\ref{sec:tkf}.
We also derive error bounds under the
big bang condition.
For $\epsilon > 0$, let $n_\epsilon < \infty$ be the smallest integer such that $\sum_{i> n_\epsilon}\pi(i) <\epsilon$ and
$
\Lambda_\epsilon = \{i \in \mathcal{S}\,:\, i \leq n_\epsilon\}.
$
Define also
$$q^*_\epsilon = \max_{i \in \Lambda_\epsilon}\, (q_i \lor 1),$$
and
$$
\Delta_{\epsilon}
=
\min_{i_1 \neq i_2 \in \Lambda_{\epsilon}}
\|
\mathbf{p}^{i_1}(h^*) 
- 
\mathbf{p}^{i_2}(h^*)
\|_\tv,
$$
which is positive under initial-state identifiability.
\begin{theorem}[Root reconstruction: error bounds]\label{thm:2}
Let $\{T^k\}_k$ and $(\mathbf{P}_t)_t$ satisfy our standing assumptions (i) and (ii) as well as the big bang condition,
and let $\pi$ be a probability distribution on $\mathcal{S}$. Fix $\epsilon > 0$ and $k \geq 1$.
Then there exist universal constants $C_0, C_1 > 0$ and an estimator $F_k$
such that
for all $s > 0$,
\begin{equation}\label{eq:thm:2:gen}
\P^{\pi}
\left[
F_k(X^k_{\partial T^{k}})
\neq 
X^k_{\rho}
\right]
< \epsilon
+
C_0\,  \Delta_\epsilon^{-2}\, q^*_\epsilon\, s
+
n_\epsilon \exp
\left(
-  C_1\, \Delta_\epsilon^2\, |\partial T^k(s)|
\right).
\end{equation}
Further, if the chain is uniform, that is, if
$
q^* 
= 
\sup_{i \in \mathcal{S}} \,(q_i \lor 1) < +\infty$,
then there exist universal constants $C^U_0, C^U_1, C^U_2 > 0$ and an estimator $F^U_k$
such that for all $s > 0$ and all $i$
\begin{equation}\label{eq:thm:2:unif}
\P^{i}
\left[
F^U_k(X^k_{\partial T^{k}})
\neq 
X^k_{\rho}
\right]
<
C^U_0\,  f_*^{-4}\, q^*\, s
+
C^U_2 f_*^{-1} \exp
\left(
-  C^U_1\, f_*^4\, |\partial T^k(s)|
\right),
\end{equation}
where $f_* = e^{-q^* h^*}$.
\end{theorem}
\noindent 
The following example gives some intuition for the terms in~\eqref{eq:thm:2:gen} and~\eqref{eq:thm:2:unif}.
\begin{example}[Two-state chain on a pinched star]
Consider the following tree $T$. The root $\rho$ is 
adjacent to a single vertex $\widetilde{\rho}$ through
an edge of length $s > 0$. The vertex $\widetilde{\rho}$
is also adjacent to $m$ vertices $x_1,\ldots, x_m$
through edges of length $h-s > 0$, where $m$ is an odd integer. Consider the $(\mathbf{P}_t)_t$-chain on $T$ with 
state space $\mathcal{S} = \{1,2\}$,
$Q$-matrix
$$
Q
=
\begin{pmatrix}
-q & q\\
q & -q
\end{pmatrix},
$$
and uniform root distribution $\pi$.
It can be shown (see e.g.~\cite{SempleSteel:03}) that under this chain
\begin{equation}\label{eq:transition-probabilities}
p_{11}(t) = \frac{1 + e^{-2qt}}{2}
\quad\text{and}\quad 
p_{12}(t) = \frac{1 - e^{-2qt}}{2}.
\end{equation}
Let $N_1$ be the number of leaves in state $1$,
let $\alpha = p_{11}(s) \in (1/2,1)$ and let $\beta = p_{12}(h-s) \in (0,1/2)$.
The estimator that maximizes the probability of correct
reconstruction is the maximum a posteriori estimate (see Lemma~\ref{lemma:info-facts}), which in this case boils down to setting $F(N_1) = 1$ if
\begin{align*}
&\frac{1}{2}\alpha \binom{m}{N_1} (1-\beta)^{N_1}\beta^{m-N_1}
+\frac{1}{2}(1-\alpha)\binom{m}{N_1}\beta^{N_1}(1-\beta)^{m-N_1}\\
&\qquad > \frac{1}{2}\alpha \binom{m}{N_1} \beta^{N_1}(1-\beta)^{m-N_1}
+\frac{1}{2}(1-\alpha)\binom{m}{N_1}(1-\beta)^{N_1}\beta^{m-N_1},
\end{align*}
and $F(N_1) = 2$ otherwise.
Observing that
$$
\alpha x + (1-\alpha) y >
\alpha y + (1-\alpha) x 
\ \iff\ 
(2 \alpha - 1)(x-y) > 0
\ \iff\ 
x > y,
$$
where we used that $\alpha > 1/2$, we get that
$F(N_1) = 1$ if and only if $N_1 > m/2$. Hence by symmetry,
for $i=1,2$,
\begin{align*}
\P^\pi[F(N_1) \neq X_\rho]
&= \P^i[F(N_1) \neq i]\\
&= \sum_{n < m/2}
\left\{
\alpha \binom{m}{n} (1-\beta)^{n}\beta^{m-n}
+(1-\alpha)\binom{m}{n}\beta^{n}(1-\beta)^{m-n}
\right\}\\
&\leq 
(1-\alpha)
+ \alpha\,
\P[N_1 < m/2\,|\,X_{\widetilde{\rho}} = 1]\\
&\leq
(1-\alpha)
+ \alpha\,
\exp\left(
-2 m \left\{\frac{1}{2}-\beta\right\}^2
\right),
\end{align*}
by Hoeffding's inequality~\cite{Hoeffding:63}.
By~\eqref{eq:transition-probabilities}, as $s \to 0$,
$$
\P^i[F(N_1) \neq i]
\leq 2 q s\, (1+o(1))
+ \exp\left(
-\frac{1}{2} m f^4 \, (1+o(1))
\right),
$$
where $f = e^{-q h}$.
\end{example}

\section{Spread}
\label{sec:spread}

We begin the proof by relating the big bang condition to a notion of spread introduced in~\cite{GascuelSteel:10}. This connection captures the basic combinatorial insights behind the proof of Theorem~\ref{thm:1}.

Let $T = (V,E,\rho,\ell)$ be a tree.
We let $\ell_{xy}$ be the length of the shared path from the root $\rho$ to the leaves $x$ and $y$. That is,
if $\path(u,v)$ denotes the set of edges on the unique path
between vertices $u$ and $v$, then we have
$$
\ell_{xy}:=\sum_{e\in \path(\rho,x)\cap\path(\rho,y)}\ell_{e}.
$$
Roughly speaking, a tree is ``well-spread'' if the average value of $\ell_{xy}$ over all pairs $(x,y)$ is small. The formal definition is as follows. 
\begin{definition}[Spread]\label{Def:Spread}
	The {\em spread} of a tree $T$ is defined as
	$$\spread(T):=\frac{\sum_{x,y} (\ell_{xy} \land 1)}{|\partial T|(|\partial T|-1)},$$
	where the summation is over all ordered pairs of distinct leaves $x\neq y$. For $\beta\in(0,\infty)$, we say that  $T$ is {\em $(1-\beta)$-spread} if $\spread(T)\leq \beta$. For a sequence of trees $\{T^k\}_k$, we say that $\{T^k\}_k$ has {\em vanishing spread} if
	$$
	\limsup_{k\to\infty} \spread(T^k) = 0.
	$$
\end{definition}
\noindent We show below that, if $\{T^k\}_k$ has vanishing spread, then the big bang condition holds.
The converse is false as illustrated in Figure~\ref{Fig:BigB2}, where the root is arbitrarily dense but the spread is dominated by a subtree away from the root.
\begin{figure}
	\begin{minipage}{.29\textwidth}
		\begin{tikzpicture}
		\node[big dot] (origin) at (0,0) {};
		\node[big dot] (A) at (-1,-1.732) {};
		\node[big dot] (L1) at (-2,-3.464)   {};
		\node[big dot] (L2) at (0,-3.464)   {};
		\node[big dot] (B) at (-0.5,-0.866)   {};
		\node[big dot] (L3) at (1,-3.464)   {};
		\node[big dot] (K1) at (-0.4,-3.464)   {};
		\node[big dot] (K2) at (-0.8,-3.464)   {};		
		\node[big dot] (K3) at (-1.2,-3.464)   {};
		\node[big dot] (K4) at (-1.6,-3.464)   {};	
		
		\draw[-] (origin) -- (L1);
		\draw[-] (A) -- (L2);
		\draw[-] (B) -- (L3);
		\draw[-] (A) -- (K1);
		\draw[-] (A) -- (K2);
		\draw[-] (A) -- (K3);
		\draw[-] (A) -- (K4);
		
		\draw (origin) node[left]{$\rho$};
		\draw (A) node[left]{$v_1$};
		\draw (B) node[left]{$v_2$};
		\end{tikzpicture}
	\end{minipage}
	\begin{minipage}{.29\textwidth}
		\begin{tikzpicture}
		\node[big dot] (origin) at (0,0) {};
		\node[big dot] (A) at (-1,-1.732) {};
		\node[big dot] (L1) at (-2,-3.464)   {};
		\node[big dot] (L2) at (0,-3.464)   {};
		\node[big dot] (B) at (-0.5,-0.866)   {};
		\node[big dot] (L3) at (1,-3.464)   {};
		\node[big dot] (C) at (-0.25,-0.433)   {};
		\node[big dot] (L4) at (1.5,-3.464)   {};
		\node[big dot] (K1) at (-0.22222,-3.464)   {};
		\node[big dot] (K2) at (-0.44444,-3.464)   {};		
		\node[big dot] (K3) at (-0.66666,-3.464)   {};
		\node[big dot] (K4) at (-0.88888,-3.464)   {};
		\node[big dot] (K5) at (-1.11111,-3.464)   {};
		\node[big dot] (K6) at (-1.33333,-3.464)   {};		
		\node[big dot] (K7) at (-1.55555,-3.464)   {};
		\node[big dot] (K8) at (-1.77777,-3.464)   {};
		
		\draw[-] (origin) -- (L1);
		\draw[-] (A) -- (L2);
		\draw[-] (B) -- (L3);
		\draw[-] (C) -- (L4);
		\draw[-] (A) -- (K1);
		\draw[-] (A) -- (K2);
		\draw[-] (A) -- (K3);
		\draw[-] (A) -- (K4);
		\draw[-] (A) -- (K5);
		\draw[-] (A) -- (K6);
		\draw[-] (A) -- (K7);
		\draw[-] (A) -- (K8);
		
		\draw (origin) node[left]{$\rho$};
		\draw (A) node[left]{$v_1$};
		\draw (B) node[left]{$v_2$};
		\draw (C) node[left]{$v_3$};
		\end{tikzpicture}
	\end{minipage}
	\begin{minipage}{.29\textwidth}
		\begin{tikzpicture}
		\node[big dot] (origin) at (0,0) {};
		\node[big dot] (A) at (-1,-1.732) {};
		\node[big dot] (L1) at (-2,-3.464)   {};
		\node[big dot] (L2) at (0,-3.464)   {};
		\node[big dot] (B) at (-0.5,-0.866)   {};
		\node[big dot] (L3) at (1,-3.464)   {};
		\node[big dot] (C) at (-0.25,-0.433)   {};
		\node[big dot] (L4) at (1.5,-3.464)   {};
		\node[big dot] (D) at (-0.125,-0.2165)   {};
		\node[big dot] (L5) at (1.75,-3.464)   {};
		\node[big dot] (E) at (-0.0625,-0.10825)   {};
		\node[big dot] (L6) at (1.875,-3.464)   {};
		\node[big dot] (K1) at (-0.22222,-3.464)   {};
		\node[big dot] (K2) at (-0.44444,-3.464)   {};		
		\node[big dot] (K3) at (-0.66666,-3.464)   {};
		\node[big dot] (K4) at (-0.88888,-3.464)   {};
		\node[big dot] (K5) at (-1.11111,-3.464)   {};
		\node[big dot] (K6) at (-1.33333,-3.464)   {};		
		\node[big dot] (K7) at (-1.55555,-3.464)   {};
		\node[big dot] (K8) at (-1.77777,-3.464)   {};
		\node[big dot] (K9) at (-0.11111,-3.464)   {};		
		\node[big dot] (K10) at (-0.33333,-3.464)   {};
		\node[big dot] (K11) at (-0.55555,-3.464)   {};		
		\node[big dot] (K12) at (-0.77777,-3.464)   {};
		\node[big dot] (K13) at (-1,-3.464)   {};
		\node[big dot] (K14) at (-1.22222,-3.464)   {};
		\node[big dot] (K15) at (-1.44444,-3.464)   {};		
		\node[big dot] (K16) at (-1.66666,-3.464)   {};
		\node[big dot] (K17) at (-1.88888,-3.464)   {};
		
		\draw[-] (origin) -- (L1);
		\draw[-] (A) -- (L2);
		\draw[-] (B) -- (L3);
		\draw[-] (C) -- (L4);
		\draw[-] (D) -- (L5);
		\draw[-] (E) -- (L6);
		\draw[-] (A) -- (K1);
		\draw[-] (A) -- (K2);
		\draw[-] (A) -- (K3);
		\draw[-] (A) -- (K4);
		\draw[-] (A) -- (K5);
		\draw[-] (A) -- (K6);
		\draw[-] (A) -- (K7);
		\draw[-] (A) -- (K8);
		\draw[-] (A) -- (K9);
		\draw[-] (A) -- (K10);
		\draw[-] (A) -- (K11);
		\draw[-] (A) -- (K12);
		\draw[-] (A) -- (K13);
		\draw[-] (A) -- (K14);
		\draw[-] (A) -- (K15);
		\draw[-] (A) -- (K16);
		\draw[-] (A) -- (K17);
		
		\draw (origin) node[left]{$\rho$};
		\draw (A) node[left]{$v_1$};
		\draw (B) node[left]{$v_2$};
		\draw (C) node[left]{$v_3$};
		\end{tikzpicture}
	\end{minipage}
	\begin{minipage}{.10\textwidth}
		$\cdots\cdots$
	\end{minipage}
	\caption{A (sub-)sequence of trees $\{T^k\}_k$ (from left to right) satisfying the big bang condition, but such that $\spread(T^k)$ does not tend to 0.}\label{Fig:BigB2}
\end{figure}
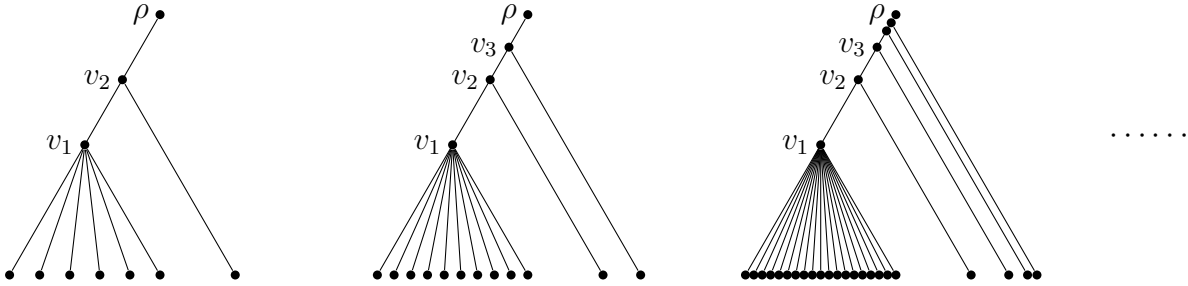 
We show however that, if the big bang condition holds,
then one can find a sequence of arbitrarily large restrictions with vanishing spread.
(Restrictions were introduced in Definition~\ref{def:restriction}.)
Our main result of this section is the following lemma. 
\begin{lemma}[Big bang and spread]
	\label{lemma:bigbang-spread}
	Let $\{T^k\}_k$ be a sequence of trees satisfying our standing assumptions  (i) and (ii). 
	The big bang condition holds if and only if there exists a nested sequence of restrictions $\widetilde{T}^{k}$ of $T^k$ such that $|\partial \widetilde{T}^k| \to \infty$ and $\{\widetilde{T}^k\}_k$ has vanishing spread.
\end{lemma}
\begin{proof}
For the {\em if} part, we argue by contradiction.
Assume the big bang condition fails and let
$\{\widetilde{T}^k\}_k$ be a nested sequence of restrictions of $\{T^k\}_k$ with vanishing spread such that $|\partial \widetilde{T}^k| \to \infty$. Then there exist $s_0\in(0,1)$, $m_0 \geq 1$ and $k_0\geq 1$ such that
$$
|\partial T^k(s_0)| = m_0, \qquad \forall k \geq k_0.
$$
Also, by the nested property, the truncation $T^k(s_0)$ remains the same for all $k\geq k_0$. We show that at least one of the subtrees of $\widetilde{T}^k$ rooted at a point in $\partial T^k(s_0)$ makes a large contribution to the spread. For $k \geq k_0$ and $z \in \partial T^k(s_0)$, let $\partial \widetilde{T}^k_{[z]}$ be the leaves of $\widetilde{T}^k$ below $z$. Then, since
$$
\sum_{z \in \partial T^k(s_0)} \left|\partial \widetilde{T}^k_{[z]}\right| = |\partial \widetilde{T}^k|,
$$
there is a $z_k \in \partial T^k(s_0)$ such that
\begin{equation}
\label{eq:m0k}
\left|\partial \widetilde{T}^k_{[z_k]}\right| \geq \left\lceil\frac{|\partial \widetilde{T}^k|}{m_0}\right\rceil.
\end{equation}
Observe that, for all distinct $x,y$ in $\partial \widetilde{T}^k_{[z_k]}$, it holds that $\ell_{xy} \geq s_0$ because the paths to $x$ and $y$ share at least the path to $z_k$. Then, counting only the contribution from $\partial \widetilde{T}^k_{[z_k]}$, we get the following bound on the spread of $\widetilde{T}^k$
$$
\spread(\widetilde{T}^k)
\geq \frac{\left|\partial \widetilde{T}^k_{[z_k]}\right|\left(\left|\partial \widetilde{T}^k_{[z_k]}\right| - 1\right) s_0}{|\partial \widetilde{T}^k| (|\partial \widetilde{T}^k| - 1)}.
$$
By~\eqref{eq:m0k},
$$
\liminf_{k \to \infty} \spread(\widetilde{T}^k)
\geq \frac{s_0}{m_0^2} > 0.
$$
As a result, $\{\widetilde{T}^k\}_k$ does not have vanishing spread.

For the {\em only if} part, assume the big bang condition holds. For every $k \geq 1$ and $s\in (0,1)$, we extract a $(1-s)$-spread restriction $\widetilde{T}^{k,s}$ of $T^k$
as follows. See Figure~\ref{Fig:Subtree} for an illustration.
\begin{figure}
	\begin{minipage}{.20\textwidth}
		\quad
	\end{minipage}
	\begin{minipage}{.4\textwidth}
		\begin{tikzpicture}
		\node[big dot] (origin) at (0,0) {};
		\coordinate (X) at (-0.75,-1.299) {};
		\node[big dot] (B) at (-0.5,-0.866)   {};
		\coordinate (L3) at (-.25,-1.299)   {};
		\node[big dot] (C) at (-0.25,-0.433)   {};
		\coordinate (L4) at (.25,-1.299)   {};
		\draw[-] (origin) -- (X);
		\draw[-] (B) -- (L3);
		\draw[-] (C) -- (L4);
		\draw (origin) node[left]{$\rho$};
		\draw (B) node[left]{$v_2$};
		\draw (C) node[left]{$v_3$};
		\draw[<->] (1,0) -- node[anchor=west] {$s$} (1,-1.299);
		\end{tikzpicture}
	\end{minipage}
	\begin{minipage}{.4\textwidth}
		\begin{tikzpicture}
		\node[big dot] (origin) at (0,0) {};
		\node[big dot] (A) at (-1,-1.732) {};
		\node[big dot] (X) at (-0.75,-1.299) {};
		\draw (X) node[anchor=north west]{$x$};
		\node[big dot,color=black!50] (L1) at (-2,-3.464)   {};
		\node[big dot,color=black!50] (L2) at (0,-3.464)   {};
		\node[big dot,color=black!50] (B) at (-0.5,-0.866)   {};
		\node[big dot,color=black!50] (L3) at (1,-3.464)   {};
		\node[big dot,color=black!50] (C) at (-0.25,-0.433)   {};
		\node[big dot,color=black!50] (L4) at (1.5,-3.464)   {};
		\node[big dot,color=black!50] (K1) at (-0.22222,-3.464)   {};
		\node[big dot,color=black!50] (K2) at (-0.44444,-3.464)   {};		
		\node[big dot,color=black!50] (K3) at (-0.66666,-3.464)   {};
		\node[big dot,color=black!50] (K4) at (-0.88888,-3.464)   {};
		\node[big dot] (K5) at (-1.11111,-3.464)   {};
		\node[big dot,color=black!50] (K6) at (-1.33333,-3.464)   {};		
		\node[big dot,color=black!50] (K7) at (-1.55555,-3.464)   {};
		\node[big dot,color=black!50] (K8) at (-1.77777,-3.464)   {};
		
		\draw[-,color=black!50] (A) -- (L1);
		\draw[-,color=black!50] (A) -- (L2);
		\draw[-,very thick] (origin) -- (A);
		\draw[-,very thick] (B) -- (L3);
		\draw[-,very thick] (C) -- (L4);
		\draw[-,color=black!50] (A) -- (K1);
		\draw[-,color=black!50] (A) -- (K2);
		\draw[-,color=black!50] (A) -- (K3);
		\draw[-,color=black!50] (A) -- (K4);
		\draw[-,very thick] (A) -- (K5);
		\draw (K5) node[below]{\small{$x^*$}};
		\draw[-,color=black!50] (A) -- (K6);
		\draw[-,color=black!50] (A) -- (K7);
		\draw[-,color=black!50] (A) -- (K8);
		
		\draw (origin) node[left]{$\rho$};
		\draw (A) node[left]{$v_1$};
		\draw (B) node[left]{$v_2$};
		\draw (C) node[left]{$v_3$};
		
		\draw[<->] (1,0) -- node[anchor=west] {$s$} (1,-1.299);
		\draw[-,dotted,very thick] (-1.25,-1.299) -- (1.25,-1.299);    
		
		\end{tikzpicture}
	\end{minipage}
	\caption{Consider again the second tree in Figure~\ref{Fig:BigB2}. On the left side, $T^k(s)$ is shown where $k=3$. On the right side, the subtree $\widetilde{T}^{k,s}$ is highlighted.}\label{Fig:Subtree}
\end{figure}
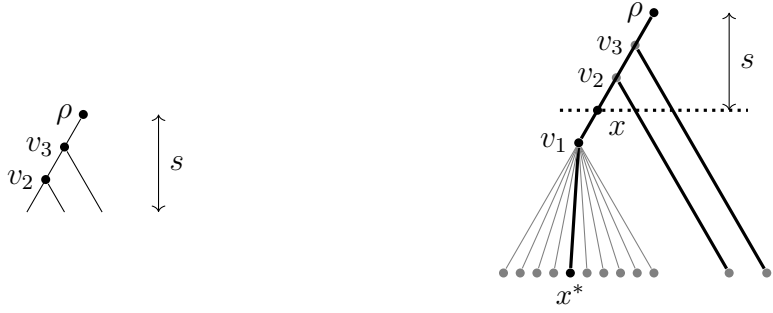
Let $\partial T^k(s) = \{z_1,\ldots,z_m\}$.
For each $z_i$, $i=1,\ldots,m$, pick an arbitrary leaf $x_i \in \partial T^k_{[z_i]}$ in the subtree $T^k_{[z_i]}$ of $T^k$ rooted at $z_i$. We let $\widetilde{T}^{k,s}$ be the
restriction of $T^k$ to $\{x_1,\ldots,x_m\}.$ 
Observe that $\widetilde{T}^{k,s}$ is $(1-s)$-spread 
because the paths to each pair of leaves 
in $\partial \widetilde{T}^{k,s}$ diverge
within $T^k(s)$. To construct a sequence of restrictions
with vanishing spread, we take a sequence of positive reals $(s_i)_{i \geq 1}$ with $s_i \downarrow 0$
and proceed as follows:
\begin{itemize}
\item Let $k_1 \geq 1$ be such that $|\partial T^{k}(s_2)| \geq 2$ for all $k > k_1$. The value $k_1$ exists under the big bang condition.	For all $k \leq k_1$, let $\widetilde{T}^k = \widetilde{T}^{k,s_1}$.

\item Let $k_2 > k_1$ be such that $|\partial T^{k}(s_3)| \geq 3$ for all $k > k_2$. The value $k_2$ exists under the big bang condition.	For all $k_1 < k \leq k_2$, let $\widetilde{T}^k = \widetilde{T}^{k,s_2}$.

\item And so forth.
\end{itemize}
By construction, for all $k_{j-1} < k \leq k_{j}$, it holds
that $\spread(\widetilde{T}^{k}) \leq s_j$ and 
$
|\partial \widetilde{T}^k| 
= |\partial T^k(s_j)| 
\geq j.
$
Thus,
$$
\limsup_{k\to \infty} \spread(\widetilde{T}^k) = 0,
$$
and
$$
\lim_k |\partial \widetilde{T}^k| = +\infty,
$$
as required.
\end{proof}

\section{Impossibility of reconstruction}
\label{sec:impossibility}

The goal of this section is to show that, in the absence of downstream disjointness, the big bang condition is necessary for consistent root reconstruction.
The following proposition implies the {\em only if} part of Theorem~\ref{thm:1}.
\begin{prop}[Impossibility of reconstruction without the big bang condition]\label{T:Imposs}
	Let $\{T^k\}_k$ and $(\mathbf{P}_t)_t$ satisfy our standing assumptions (i) and (ii), and let $\pi$ be a probability distribution on $\mathcal{S}$. 
	Assume that neither downstream disjointness nor the big bang condition hold.
	Then consistent reconstruction of the root state is impossible, in the sense that there exists an $\epsilon > 0$ such that for all $k \geq 1$
	\begin{equation}\label{E:Imposs}
	\sup_{F_k} \P^{\pi}\left[F_k(X^k_{\partial T^k})=X^k_{\rho}\right] \leq 1-\epsilon,
	\end{equation}
	where the supremum is over all 
	root estimators
	$F_k : \mathcal{S}^{\partial T^k} \to \mathcal{S}$.
\end{prop}

\subsection{Information-theoretic bounds}

To prove Proposition~\ref{T:Imposs}, 
we need some information-theoretic bounds that relate the
best achievable reconstruction probability to the total variation distance between the conditional distributions of pairs of initial states. Our first bound says roughly that the reconstruction probability is only as good as the worst total variation distance. Our second bound shows that a good reconstruction probability can be obtained from selecting a subset of initial states with high prior probability whose corresponding conditional distributions have ``little overlap.''
See e.g.~\cite[Chapter 2]{CoverThomas:06} and~\cite{SteelSzekely:99,SteelSzekely:02} for some related results.
\begin{lemma}[Information-theoretic bounds]
	\label{lemma:info-facts}
	Let $Y_0$ and $Y_1$ be random variables taking values in the countable spaces $\mathcal{Y}_0$ and $\mathcal{Y}_1$ respectively. Let $\mu_0$ denote the distribution of $Y_0$ and let $\mu_1^i$ denote the distribution of $Y_1$ conditioned on $\{Y_0 = i\}$. 
	\begin{enumerate}
		\item {\em (Reconstruction upper bound)} 
		It holds that
		\begin{equation}\label{eq:recon-upper}
		\sup_{f:\mathcal{Y}_1 \to \mathcal{Y}_0} \P[f(Y_1) = Y_0] \leq 
		1 -
		\sup_{i_1 \neq i_2 \in \mathcal{Y}_0}
		\left\{
		\mu_0(i_1) \land \mu_0(i_2) 
		\left[
		1 
		- 
		\|\mu_1^{i_1} - \mu_1^{i_2}\|_\tv
		\right]
		\right\}.
		\end{equation}
		
		\item {\em (Reconstruction lower bound)} 
		For any $\Lambda \subseteq \mathcal{Y}_0$, 
		it holds that
		\begin{equation}\label{eq:recon-lower}
		\sup_{f:\mathcal{Y}_1 \to \mathcal{Y}_0} \P[f(Y_1) = Y_0] \geq 
		\sum_{i \in \Lambda} \mu_0(i)-
		\sum_{i_1 \neq i_2 \in \Lambda}
		\left\{
		\mu_0(i_1) \lor \mu_0(i_2) 
		\left[
		1 
		- 
		\|\mu_1^{i_1} - \mu_1^{i_2}\|_\tv
		\right]
		\right\}.
		\end{equation}		
		
	\end{enumerate}
\end{lemma}
\begin{proof}
	For both bounds, our starting point is the formula
	\begin{equation}\label{eq:startig-point}
	\P[f(Y_1) = Y_0]
	=
	\sum_{i \in \mathcal{Y}_0} \sum_{j \in \mathcal{Y}_1}
	{\bf 1}_{\{f(j) = i\}}\, \mu_1^i(j)\, \mu_0(i).
	\end{equation}
	We will also need the following alternative
	expression for total variation distance
	\begin{equation}
	\label{eq:tv-min}
	\|\mu_1^{i_1} - \mu_1^{i_2}\|_\tv
	= 
	1 
	- 
	\sum_{j \in \mathcal{Y}_1} 
	\mu_1^{i_1}(j) \land \mu_1^{i_2}(j),
	\end{equation}
	which follows from the last equality in~\eqref{eq:def-tv}.
	
	To derive~\eqref{eq:recon-upper}, observe first that
	by~\eqref{eq:startig-point} for any $f$
	\begin{align*}
	\P[f(Y_1) = Y_0]
	\leq
	\sum_{j \in \mathcal{Y}_1}
	\sup_{i \in \mathcal{Y}_0} 
	\mu_1^i(j)\, \mu_0(i)
	=
	\P[f^*(Y_1) = Y_0],
	\end{align*}
	where $f^*$ is a maximum a posteriori estimate
	\begin{equation}\label{eq:map}
	f^*(j) 
	\in \underset{i \in \mathcal{Y}_0}{\arg\max} \left\{ \mu_1^i(j)\, \mu_0(i)\right\}.
	\end{equation}
	This is well-known; see e.g.~\cite[Chapter 4]{Lehmann:98}.
	The maximum above is indeed attained because 
	$\mu_1^i(j)\, \mu_0(i)$ is summable over $i$, and therefore must converge to $0$. Then, by~\eqref{eq:startig-point} applied to $f = f^*$,
	for any $i_1 \neq i_2 \in \mathcal{Y}_0$
	\begin{align*}
	\P[f^*(Y_1) = Y_0]
	&\leq
	\sum_{i \in \mathcal{Y}_0} \sum_{j \in \mathcal{Y}_1}
	\mu_1^i(j)\, \mu_0(i)
	- 
	\sum_{j \in \mathcal{Y}_1}
	[\mu_1^{i_1}(j)\, \mu_0(i_1)]
	\land 
	[\mu_1^{i_2}(j)\, \mu_0(i_2)]\\
	&\leq 
	1
	-
	\left(\mu_0(i_1)
	\land 
	\mu_0(i_2)\right)
	\sum_{j \in \mathcal{Y}_1}
	\mu_1^{i_1}(j)
	\land 
	\mu_1^{i_2}(j).
	\end{align*}
	Bound~\eqref{eq:recon-upper} then follows
	from~\eqref{eq:tv-min} and taking
	a supremum over $i_1 \neq i_2$.
	
	For~\eqref{eq:recon-lower}, 
	define the approximate maximum a posteriori estimator
	$$
	f^*_\Lambda(j) 
	\in \underset{i \in \Lambda}{\arg\max} \left\{ \mu_1^i(j)\, \mu_0(i)\right\},
	$$
	where note that, this time, the supremum
	is over $\Lambda$ only.
	Then~\eqref{eq:startig-point} applied to $f = f^*_\Lambda$ implies
	\begin{align*}
	\P[f^*_\Lambda(Y_1) = Y_0]
	&= 
	\sum_{i \in \Lambda} \sum_{j \in \mathcal{Y}_1}
	{\bf 1}_{\{f^*_\Lambda(j) = i\}}\, \mu_1^i(j)\, \mu_0(i)\\
	&=
	\sum_{i \in \Lambda} \sum_{j \in \mathcal{Y}_1}
	\mu_1^i(j)\, \mu_0(i)
	- 
	\sum_{j \in \mathcal{Y}_1}
	\sum_{i \neq f^*_\Lambda(j)}
	\mu_1^{i}(j)\, \mu_0(i)\\
	&\geq 
	\sum_{i \in \Lambda} \mu_0(i)
	-
	\sum_{j \in \mathcal{Y}_1}
	\sum_{i_1 \neq i_2 \in \Lambda}
	[\mu_1^{i_1}(j)\mu_0(i_1)]
	\land 
	[\mu_1^{i_2}(j)\mu_0(i_2)]\\
	&\geq 
	\sum_{i \in \Lambda} \mu_0(i)
	-
	\sum_{i_1 \neq i_2 \in \Lambda}
	\left\{
	\left(\mu_0(i_1)\lor \mu_0(i_2)\right)
	\sum_{j \in \mathcal{Y}_1}
	\mu_1^{i_1}(j)
	\land 
	\mu_1^{i_2}(j)
	\right\}.
	\end{align*}
	By~\eqref{eq:tv-min}, that implies~\eqref{eq:recon-lower}
	and concludes the proof.
\end{proof}

\subsection{Characterization of consistent root reconstruction}

From Lemma~\ref{lemma:info-facts}, we obtain 
a characterization of consistent root reconstruction
in terms of total variation. This characterization is key to proving both directions of Theorem~\ref{thm:1}.
Recall that $\mathcal{L}^i_{T}$ was defined
in~\eqref{eq:def-lit} as the leaf distribution on $T$ given root state $i$.
\begin{lemma}[Consistent root reconstruction: characterization]
\label{lem:charac}
Let $\{T^k\}_k$ and $(\mathbf{P}_t)_t$ satisfy our standing assumptions (i) and (ii), and let $\pi$ be a probability distribution on $\mathcal{S}$. Then there exists a sequence of
root estimators that is consistent for $\{T^k\}_k$, $(\mathbf{P}_t)_t$ and $\pi$ if and only if
for all $i \neq j \in \mathcal{S}$ such that $\pi(i) \land \pi(j) > 0$
\begin{equation}\label{eq:charac-condition}
\liminf_{k \to \infty}
\|\mathcal{L}^i_{T^k} - \mathcal{L}^j_{T^k}\|_\tv 
= 1.
\end{equation}
\end{lemma}
\begin{proof}
For the {\em only if} part, assume by contradiction that
there is $i_1 \neq i_2 \in \mathcal{S}$ with $\pi(i_1) \land \pi(i_2) > 0$, $\epsilon > 0$ and $k_0 \geq 1$ 
such that
\begin{align*}
\|\mathcal{L}^{i_1}_{T^k} - \mathcal{L}^{i_2}_{T^k}\|_\tv 
\leq 1 - \epsilon,
\end{align*}
for all $k \geq k_0$. By~\eqref{eq:recon-upper} in Lemma~\ref{lemma:info-facts}, for all $k \geq k_0$ and any root estimator $F_k$
\begin{align*}
\P^\pi[F_k(X^k_{\partial T^k}) = X^k_\rho]
\leq 
1 
-
[\pi(i_1)\land\pi(i_2)]
\,\epsilon
< 1. 
\end{align*}
That proves that consistent root estimation is not possible.

For the {\em if} part, assume~\eqref{eq:charac-condition} holds. 
Fix $\epsilon > 0$ and
let $1 \leq n_\epsilon < +\infty$ be the smallest integer
such that
\begin{equation}
\label{eq:charac-lambdaeps}
\sum_{i \leq n_\epsilon} \pi(i) > 1 - \epsilon,
\end{equation}
and let $\Lambda_\epsilon = \{i\,:\,i\leq n_\epsilon\}$.
Applying~\eqref{eq:recon-lower} in Lemma~\ref{lemma:info-facts} with $\Lambda = \Lambda_\epsilon$, we get by~\eqref{eq:charac-condition} and~\eqref{eq:charac-lambdaeps}
\begin{align*}
\sup_{F_k} \P^\pi[F_k(X^k_{\partial T^k}) = X^k_\rho]
&\geq 
1-\epsilon
- \sum_{i_1 \neq i_2 \in \Lambda_\epsilon}
\left\{
\pi(i_1)
\lor
\pi(i_2)
\left[
1 
-
\|
\mathcal{L}^{i_1}_{T^k} - \mathcal{L}^{i_2}_{T^k}
\|_\tv 
\right]
\right\}\\
&\to 1 - \epsilon,
\end{align*}
as $k \to \infty$.
Because $\epsilon$ is arbitrary, we have shown
that a sequence of maximum posteriori estimates is consistent
for $\{T^k\}_k$, $(\mathbf{P}_t)_t$ and $\pi$.
\end{proof}

\subsection{Proof of Proposition~\ref{T:Imposs}}

We now prove our main result of this section.

\smallskip

\begin{proof}[Proof of Proposition \ref{T:Imposs}]
	Let $\{T^k\}_k$ and $(\mathbf{P}_t)_t$ satisfy our standing assumptions (i) and (ii), and let $\pi$ be a probability distribution on $\mathcal{S}$. 
	Assume that $\{T^k\}_k$ satisfies neither downstream disjointness nor the big bang condition.
	Then, as we argued in the proof of Lemma~\ref{lemma:bigbang-spread}, there exist $s_0\in(0,\infty)$ and $k_0\geq 1$ such that the truncation $T^k(s_0)$ remains unchanged for all $k\geq k_0$. 	Since downstream disjointness fails and $\ell_u > 0$ for all $u \in \partial T^k$ (by the positivity assumption on $\ell$), there are
	$i_1 \neq i_2$ with $\pi(i_1) > 0$ and $\pi(i_2) > 0$ 
	such
	that the supports of $\P^{i_1}[X^k_{\partial T^k(s_0)} \in \cdot\, ]$ and $\P^{i_2}[X^k_{\partial T^k(s_0)} \in \cdot\, ]$ have a non-empty intersection. This holds for all $k$ and implies that
	$$
	\|\P^{i_1}[X^k_{\partial T^k(s_0)} \in \cdot\, ] - \P^{i_2}[X^k_{\partial T^k(s_0)} \in \cdot\, ]\|_\tv < 1.
	$$ 
	Because $T^k(s_0)$ is unchanged after $k_0$,
	it follows that
	\begin{equation}\label{eq:liminf-truncation}
	\liminf_{k \to \infty} \|\P^{i_1}[X^k_{\partial T^k(s_0)} \in \cdot\, ] - \P^{i_2}[X^k_{\partial T^k(s_0)} \in \cdot\, ]\|_\tv < 1.
	\end{equation}
	Finally we observe that, by the triangle inequality and the conditional independence of $X^k_\rho$ and $X^k_{\partial T^k}$ given $X^k_{\partial T^k(s_0)}$,
	we get 
	\begin{align}
	&\|\mathcal{L}^{i_1}_{T^k} - \mathcal{L}^{i_2}_{T^k}\|_\tv\nonumber\\
	&\qquad =
	\frac{1}{2}\sum_{\sigma \in \mathcal{S}^{\partial T^k}}
	\left|
	\mathcal{L}^{i_1}_{T^k}(\sigma)
	- \mathcal{L}^{i_2}_{T^k}(\sigma)
	\right|\nonumber\\
	&\qquad =
	\frac{1}{2}\sum_{\sigma \in \mathcal{S}^{\partial T^k}}
	\left|
	\sum_{\tau \in \mathcal{S}^{\partial T^k(s_0)}}
	\P[X^k_{\partial T^k} = \sigma\,|\,X^k_{\partial T^k(s_0)} = \tau]	
	\left[
	\P^{i_1}[X^k_{\partial T^k(s_0)} = \tau]
	- \P^{i_2}[X^k_{\partial T^k(s_0)} = \tau]
	\right]
	\right|\nonumber\\
	&\qquad \leq \|\P^{i_1}[X^k_{\partial T^k(s_0)} \in \cdot\, ] - \P^{i_2}[X^k_{\partial T^k(s_0)} \in \cdot\, ]\|_\tv.\label{eq:loss-tv}
	\end{align}
	Combining this inequality with~\eqref{eq:liminf-truncation}
	shows by Lemma~\ref{lem:charac} that
	consistent root estimation is not possible in this case.
	That concludes the proof.
\end{proof}

\section{Consistent root reconstruction}\label{S:SAR}
\label{sec:consistency}

In this section, 
we prove the {\em if} part of Theorem~\ref{thm:1}.
Observe first that, under downstream disjointness,
the result is immediate. Let $u \in \partial T^1$
and $I = \{i\,:\, \pi(i) > 0\}$. 
Note that, by the nested property, $u \in \partial T^k$ for all $k$. Then, let $F_k(X^k_{\partial T^k})$
be the state in $I$ from which $X^k_u$ is reachable. 
Downstream disjointness ensures that such a state exists and is unique. We then have $\P^\pi[F_k(X^k_{\partial T^k}) = X^k_\rho] = 1$, proving consistency in that case.

Here we show that the big bang condition also suffices for consistent root reconstruction.
We use the characterization in Lemma~\ref{lem:charac} to reduce the problem to pairs of initial states.
Our strategy is then to extract a ``well-spread'' subtree of $T^k$, as we did in the proof of Lemma~\ref{lemma:bigbang-spread}, and generalize results of~\cite{GascuelSteel:10} on root reconstruction for well-spread trees. 
Formally we prove the following proposition,
which together with Lemma~\ref{lem:charac} and the argument above in the downstream disjointness case, implies
the {\em if} part of Theorem~\ref{thm:1}.
\begin{prop}[Reconstruction under the big bang condition] \label{T:Consistent} 
	Let $\{T^k\}_k$ and $(\mathbf{P}_t)_t$ satisfy our standing assumptions (i) and (ii), and let $\pi$ be a probability distribution on $\mathcal{S}$. 
	Assume that $\{T^k\}_k$ satisfies the big bang condition.
	Then for all $i \neq j \in \mathcal{S}$ such that $\pi(i) \land \pi(j) > 0$
	\begin{equation}\label{eq:charac-big-bang}
	\liminf_{k \to \infty}
	\|\mathcal{L}^i_{T^k} - \mathcal{L}^j_{T^k}\|_\tv 
	= 1.
	\end{equation}
\end{prop}

\subsection{Well-spread restriction}
\label{sec:root-estimator}

We will use the following construction. We extract a well-spread restriction of $T^k$ and stretch the leaf edges to enforce that all leaves are at the same distance from the root.
Fix $k \geq 1$ and $s > 0$.
Recall that $h^*$ is a (uniform) bound on the height
of the trees.
\begin{itemize}
	\item {\bf Step 1: Well-spread restriction.} 
	By Lemma~\ref{lemma:bigbang-spread}, there exists a a nested sequence of restrictions with vanishing spread.
	Let $\widetilde{T}^{k,s}$ be the restriction of $T^k$ constructed in the proof of Lemma~\ref{lemma:bigbang-spread}. Recall that $\widetilde{T}^{k,s}$ is $(1-s)$-spread and
	has $|\partial T^k(s)|$ leaves. 
	
	\item {\bf Step 2: Stretching.} We then modify $\widetilde{T}^{k,s}$ to make all leaves be at distance $h^*$ from the root as follows. For each leaf $x \in \partial \widetilde{T}^{k,s}$, we extend the corresponding leaf edge by $h^* - \ell_x$ and
	run the $\mathbf{P}_t$-chain started at $X^k_x$ for time $h^* - \ell_x$. We then let
	$\widehat{T}^{k,s}$ be the resulting tree and assign the states generated above along the extensions. 
	Observe that $\widehat{T}^{k,s}$, like $\widetilde{T}^{k,s}$, is $(1-s)$-spread and
	has $|\partial T^k(s)|$ leaves.
\end{itemize}
\noindent Let $N^{(k)}_j$ be the number of leaves of the stretched restriction $\widehat{T}^{k,s}$ that are in state $j\in\mathcal{S}$ and
let $\mathbf{N}^{k,s}=(N^{k,s}_1,N^{k,s}_2,\cdots)$.
Denote by $\mathcal{M}^i_{\widehat{T}^{k,s}}$ the law
of $\mathbf{N}^{k,s}$ when the root state is $i$.
By a computation similar to~\eqref{eq:loss-tv}, 
by the conditional independence of
$\mathbf{N}^{k,s}$ and $X^k_\rho$
given $X^k_{\partial T^k}$,
we have that
$$
\|\mathcal{M}^i_{\widehat{T}^{k,s}}
- \mathcal{M}^j_{\widehat{T}^{k,s}}\|_\tv
\leq
\|\mathcal{L}^i_{T^k} - \mathcal{L}^j_{T^k}\|_\tv.
$$
Therefore, Proposition~\ref{T:Consistent} follows
from the following lemma.
\begin{lemma}[Separation of state frequencies on stretched restrictions]
\label{lemma:T:Consistent}
Consider the setting of Proposition~\ref{T:Consistent}
and let $\mathcal{M}^i_{\widehat{T}^{k,s}}$ be as defined above.
Then,
$$
\sup_{s > 0}
\liminf_{k \to \infty}
\|\mathcal{M}^i_{\widehat{T}^{k,s}}
- \mathcal{M}^j_{\widehat{T}^{k,s}}\|_\tv
=1.
$$
\end{lemma}

When all leaves of $T^k$ are assumed to be at the same distance from the root,
$T^k$ is said to be ultrametric (see e.g.~\cite[Chapter 7]{SempleSteel:03}). 
Here
we do not make this assumption on $T^k$. Instead we enforce it
artificially through the stretching in Step 2. The reason we do this is that our proof relies on initial-state identifiability which, by~\eqref{eq:ident-tv}, implies
\begin{equation}
\label{eq:ident-tv-hstar}
\|
\mathbf{p}^{i}(h^*) - \mathbf{p}^{j}(h^*)\|_\tv > 0,
\qquad \forall i\neq j \in \mathcal{S}.
\end{equation}
 In contrast, it may not be the case that the expected state frequencies at $\partial \widetilde{T}^{k,s}$, that is,
$$
\frac{1}{|\partial T^k(s)|}\sum_{x \in \partial \widetilde{T}^{k,s}} \mathbf{p}^i(\ell_x),
$$
uniquely characterize the root state $i$.

\subsection{Variance bound}

The proof of Lemma~\ref{lemma:T:Consistent} relies on the following variance bound, which generalizes a result of~\cite[Proof of Lemma 3.2]{GascuelSteel:10}. Recall the definition of $q_i$ in~\eqref{eq:def-qi}. 
\begin{lemma}[Variance bound]\label{L:var_nj}
Let $T = (V,E,\rho,\ell)$ be a tree and
let $(X_\gamma)_{\gamma \in \Gamma_T}$ be a $\mathbf{P}_t$-chain on $T$.
Let $N_j$ be the number of leaves of $T$ in state $j\in \mathcal{S}$. Then for all $i,j\in \mathcal{S}$,
	\begin{align}
	\var_{i}(N_j) &\leq \frac{1}{4} |\partial T| +2 (q_{i}\lor 1)\,\spread(T)\,|\partial T|^2, \label{var_nj}
	\end{align}
where we denote by $\var_i$ the variance under $\P^i$. 
\end{lemma}
\begin{proof}
	Let $\theta^j_x$ be the indicator random variable for the event ``leaf $x$ is in state $j$.'' Then 
    $$
    N_j=\sum_{x\in \partial T}\theta^j_x,
    $$ 
    and, hence,
	\begin{align}
	\var_{i}(N_j)&=\sum_{x}\var_i(\theta^j_x) + \sum_{x\neq y}\cov_i(\theta^j_x,\;\theta^j_y).\label{var_nj2}
	\end{align}
	Because $\theta^j_x\in \{0,1\}$, we have 
    $$
    \var_i(\theta^j_x) = \P^i[\theta^j_x = 1] (1-\P^i[\theta^j_x = 1]) \leq 1/4,
    $$ 
leading to the first term on the RHS of~\eqref{var_nj}.
For $x\neq y$, we have that
	\begin{align*}
	 \cov_i(\theta^{j}_x,\;\theta^{j}_y)
	 &=\E^i\big[\big(\theta^{j}_x-p_{ij}(\ell_x)\big)\,\big(\theta^{j}_y-p_{ij}(\ell_y)\big)\big] \\
	 &=\sum_{k\in \mathcal{S}}p_{ik}(\ell_{xy})\,\big(p_{kj}(\ell_x-\ell_{xy})-p_{ij}(\ell_x)\big)\,\big(p_{kj}(\ell_y-\ell_{xy})-p_{ij}(\ell_y)\big),
	\end{align*}
	which is obtained by conditioning on the state 
    at the divergence point between the paths from the root to $x$ and $y$. Splitting the sum according to whether $k=i$, we have 
	\begin{align}	|\cov_i(\theta^{j}_x,\,\theta^{j}_y)|
 	&\leq\left|p_{ij}(\ell_x-\ell_{xy})-p_{ij}(\ell_x)\right| +  \sum_{k\neq i}p_{ik}(\ell_{xy}) \label{|Cov|1}\\    
	&\leq 2 [(q_{i} \ell_{xy} )\land 1].\label{|Cov|2}
	\end{align}
	To see inequality~\eqref{|Cov|2}, note that the second term on the RHS  of \eqref{|Cov|1} is bounded above by the probability that
	the state is changed at least once along the shared path from the root to $x$ and $y$, which is equal to $1-\exp{(- q_{i}\ell_{xy})} \leq (q_{i}\ell_{xy}) \land 1$ (see e.g.~\cite[Chapter 2]{Liggett:10}). For the first term, 
    the Chapman-Kolmogorov equations (see e.g.~\cite[Chapter 2]{Liggett:10}) imply that, for all $t \geq 0$ and $\delta > 0$, 
$$
p_{ij}(t + \delta) - p_{ij}(t)
=
\sum_{k} p_{ik}(\delta) p_{kj} (t) 
- p_{ij}(t)$$ 
so that
\begin{align*}
p_{ij}(t + \delta)
- p_{ij}(t)
\leq \sum_{k\neq i} p_{ik} (\delta)
= 1 - p_{ii}(\delta)
\leq 1 - \exp(-q_i \delta),
\end{align*}
and
\begin{align*}
p_{ij}(t + \delta)
- p_{ij}(t)
\geq -(1-p_{ii}(\delta)) p_{ij}(t)
\geq - (1 - \exp(-q_i \delta)).
\end{align*}
The proof is complete in view of \eqref{var_nj2}
and the definition of the spread.
\end{proof}

\subsection{Proof of Lemma~\ref{lemma:T:Consistent}}

\begin{proof}[Proof of Lemma~\ref{lemma:T:Consistent}]
It suffices to find a sequence of events
$\mathcal{A}_k$, $k \geq 1$, depending only on $\mathbf{N}^{k,s}$ such that
$$
\sup_{s > 0}\liminf_{k \to \infty}
\P^i[\mathcal{A}^{k,s}]
= 1
\qquad\text{and}\qquad
\inf_{s > 0}
\limsup_{k \to \infty} 
\P^j[\mathcal{A}^{k,s}]
=0,
$$
that is,
a sequence of events asymptotically
likely under $\mathcal{M}^i_{\widehat{T}^{k,s}}$ but unlikely under $\mathcal{M}^j_{\widehat{T}^{k,s}}$.

Consider the norm $\|\cdot\|_*$ defined as 
\begin{equation*} 
\|\mathbf{v}\|_{*}:=\sum_{i=1}^{|\mathcal{S}|}2^{-i}|v_i|,
\end{equation*} 
for $\mathbf{v} = (v_1, v_2, \ldots)$.
We claim that~\eqref{eq:ident-tv-hstar} is equivalent to
\begin{equation}
\Delta^*_{i,j}
:=
\|
\mathbf{p}^{i}(h^*) - \mathbf{p}^{j}(h^*)\|_* > 0.\label{eq:ident-star}
\end{equation}
Indeed, by the definition of the norms, we have $\|\cdot\|_{*}\leq \|\cdot\|_\tv$. For the other direction, note that, for any $\delta>0$, there exists $M$ such that $\sum_{k > M}2^{-k}<\delta/2$ and so $\|\mu-\nu\|_\tv \leq \delta/2+ 2^{M}\|\mu-\nu\|_{*}$ for any probability distributions $\mu$ and $\nu$.
We consider the following events
$$
\mathcal{A}^{k,s}
=
\left\{
\left\|
\frac{\mathbf{N}^{k,s}}{|\partial T^k(s)|} - \mathbf{p}^{i}(h^*)
\right\|_{*}
<
\frac{\Delta^*_{i,j}}{2}
\right\}.
$$

Because $\widehat{T}^{k,s}$ is $(1-s)$-spread, the variance bound in Lemma~\ref{L:var_nj} implies that for $i,j \in \mathcal{S}$
\begin{align}
\var_{i}\left(N^{k,s}_j\right) 
&\leq \frac{|\partial T^k(s)|}{4}+2\,(q_{i}\lor 1)\,s\,|\partial T^k(s)|^2,\label{var_nj3}
\end{align}
By the Cauchy-Schwarz inequality and \eqref{var_nj3}, 
\begin{align}
\E^i\left[
\left\|\frac{\mathbf{N}^{k,s}}{|\partial T^k(s)|} - \mathbf{p}^{i}(h^*) \right\|^2_{*}
\right]
&=
\E^i\left[\left(\sum_{j=1}^{|\mathcal{S}|}2^{-j}\left|\frac{N^{k,s}_j}{|\partial T^k(s)|}-p_{ij}(h^*)\right|\right)^2\right] \notag \\
&\leq
\left(\sum_{j=1}^{|\mathcal{S}|}2^{-j}\right)
\left(
\sum_{j=1}^{|\mathcal{S}|}2^{-j}\var_i\left(\frac{N^{k,s}_j}{|\partial T^k(s)|}\right)
\right)
\notag\\
&\leq\frac{1}{4|\partial T^k(s)|}+2 (q_{i}\lor 1)s.\label{Lem3.2} 
\end{align}
By Chebyshev's inequality (see e.g.~\cite{Durrett:10}), 
\begin{align}
\P^i\left[
\left\|
\frac{\mathbf{N}^{k,s}}{|\partial T^k(s)|} - \mathbf{p}^{i}(h^*)
\right\|_* 
\geq  
\frac{\Delta^*_{i,j}}{2}\right]
&\leq \frac{4}{(\Delta^*_{i,j})^2}\,\E_i\left[
\left\|\frac{\mathbf{N}^{k,s}}{|\partial T^k(s)|} - \mathbf{p}^{i}(h^*) \right\|^2_{*}
\right]\nonumber\\
&\leq  
\frac{4}{(\Delta^*_{i,j})^2}\left[
\frac{1}{4|\partial T^k(s)|}+2 (q_{i}\lor 1)s
\right],\label{Chebyshev}
\end{align}
where we used~\eqref{Lem3.2}.
By the big bang condition and~\eqref{eq:ident-star}, taking $k \to +\infty$ and then $s \to 0$, we get
$$
\inf_{s > 0}
\limsup_{k \to \infty}
\P^i[(\mathcal{A}^{k,s})^c]
= 0.
$$

Similarly, noting that by the triangle inequality and
the definition of $\Delta^*_{i,j}$,
$$
\P^j\left[
\left\|
\frac{\mathbf{N}^{k,s}}{|\partial T^k(s)|} - \mathbf{p}^{i}(h^*)
\right\|_* 
< 
\frac{\Delta^*_{i,j}}{2}\right]
\leq
\P^j\left[
\left\|
\frac{\mathbf{N}^{k,s}}{|\partial T^k(s)|} - \mathbf{p}^{j}(h^*)
\right\|_* 
\geq  
\frac{\Delta^*_{i,j}}{2}\right],
$$ 
we also get that
$$
\inf_{s > 0}
\limsup_{k \to \infty}
\P^j[\mathcal{A}^{k,s}]
= 0.
$$
The proof is complete. 
\end{proof}

\section{Error bounds}
\label{section:rate}

The proof of Lemma~\ref{lemma:T:Consistent} actually implies an explicit
bound on the error probability (see~\eqref{Chebyshev}). 
That bound decays like the inverse of $|\partial T^k(s)|$. This is far from best possible: take for instance the star tree where, by conditional independence of the leaf states given the root state, one would expect an exponential inequality. Here we give an improved 
bound on the achievable error probability which
decays exponentially in $|\partial T^k(s)|$. We also
express this bound in terms of the more natural total variation distance. 

Our main result is the following
proposition, which implies the first part of Theorem~\ref{thm:2}.
(The second part of the theorem is proved in Section~\ref{sec:uniform}.)
For $\epsilon > 0$, recall that $n_\epsilon < \infty$ be the smallest integer such that $\sum_{i> n_\epsilon}\pi(i) <\epsilon$ and that 
$
\Lambda_\epsilon = \{i \in \mathcal{S}\,:\, i \leq n_\epsilon\},
$
$$q^*_\epsilon = \max_{i \in \Lambda_\epsilon}\, (q_i \lor 1),$$
and
$$
\Delta_{\epsilon}
=
\min_{i_1 \neq i_2 \in \Lambda_{\epsilon}}
\|
\mathbf{p}^{i_1}(h^*) 
- 
\mathbf{p}^{i_2}(h^*)
\|_\tv.
$$
\begin{prop}[Achievable error bound]
\label{prop:conv-rate}
Fix $\epsilon > 0$ and $k \geq 1$. Then there exist universal constants $C_0, C_1 > 0$ and an estimator $F_k$
such that the following holds.
For all $s > 0$,
\begin{equation*}
\P^{\pi}
\left[
F_k(X^k_{\partial T^{k}})
\neq 
X^k_{\rho}
\right]
< \epsilon
+
C_0\,  \Delta_\epsilon^{-2}\, q^*_\epsilon\, s
+
n_\epsilon \exp
\left(
-  C_1\, \Delta_\epsilon^2\, |\partial T^k(s)|
\right).
\end{equation*}
\end{prop}

\subsection{Deviation of frequencies}
\label{sec:bound-deviation}

To prove Proposition~\ref{prop:conv-rate}, we 
devise a root estimator (described in details in the next subsection) based on the combinatorial construction of Section~\ref{sec:root-estimator}. 
Fix $k \geq 1$ and $s > 0$.
Given the leaf states $X^k_{\partial T^k}\in \mathcal{S}^{\partial T^{k}}$ of the original tree $T^k$, we extract the subtree $\widetilde{T}^{k,s}$, run a simulation of the $\mathbf{P}_t$-chain on the extended tree $\widehat{T}^{k,s}$, and treat the 
leaf states of $\widehat{T}^{k,s}$ as the observed leaf states. 
For a subset $\mathcal{A} \subseteq \mathcal{S}$, let $N^{k,s}_\mathcal{A}$ be the number of leaves of $\widehat{T}^{k,s}$ whose state is in $\mathcal{A}$.
The proof of Proposition~\ref{prop:conv-rate} requires a bound on the deviation of $N^{k,s}_\mathcal{A}$. To obtain such a bound, we proceed by first controlling the number of points in $\partial T^k(s)$ whose state coincides with the root state.

Let $i$ be state at the root.
For any vertex $v$, let $Z_v$ be $1$ if the state at $v$ is $i$, and let $Z_v$ be $0$ otherwise.
Let $\mathcal{W}_i$ be those vertices in $\partial T^k(s)$
in state $i$. In particular 
$$
S_i = |\mathcal{W}_i| = \sum_{x \in \partial T^k(s)} Z_x.
$$
Let $\widehat{N}_\mathcal{A}$ be the number of descendant leaves of $\mathcal{W}_i$ in $\widehat{T}^{k,s}$ whose states are in $\mathcal{A}$. 
We also let $m = |\partial T^k(s)|$.
Then, we can bound $N^{k,s}_\mathcal{A}$ as follows
\begin{equation}
\widehat{N}_\mathcal{A}
\leq
N^{k,s}_\mathcal{A} 
\leq  \widehat{N}_\mathcal{A} + m-S_i.\label{eq:approx-nkj}
\end{equation}
Conditioned on $S_i$, note
that $\widehat{N}_\mathcal{A}$ is a binomial random variable, specifically, $\mathrm{Bin}(S_i,\,p_{i\mathcal{A}}(h^*-s))$, where $p_{i\mathcal{A}}(t)$ denotes the probability that the state is in $\mathcal{A}$ given that initially it is $i$.
To bound the probability that
$N^{k,s}_\mathcal{A}$ is close to its expectation, we
argue in two steps. We first bound the probability
that $S_i$ itself is close to its expectation, then we apply a concentration inequality to $N^{k,s}_\mathcal{A}$ conditioned on that event.
\begin{lemma}[Control of $S_i$]
\label{lem:control-si}
Define the event
$$
\mathcal{E}^0_\delta 
=
\left\{
\left|
S_i-\E^i[S_i] 
\right|
>  
\delta m
\right\},
$$
where
\begin{align}\label{eq:e-si}
\E^i[S_i] 
= p_{ii}(s)\,m.
\end{align}
Then, we have the bound
\begin{align}
\P^i
\left(
\mathcal{E}^0_\delta 
\right) 
&\leq 
\frac{1- e^{-q_i s}}{\delta^2}.\label{eq:error-0}
\end{align}
\end{lemma}
\begin{proof}
We use Chebyshev's inequality to control the deviation of $S_i$. 
By the Cauchy-Schwarz inequality,
the variance of $S_i$ is bounded by
\begin{align*}
\var_i[S_i]
&= \var_i\left[\sum_{x \in \partial T^k(s)} Z_x\right]\\
&= \sum_{x \in \partial T^k(s)}\sum_{y \in \partial T^k(s)} \E^i\left[(Z_x - p_{ii}(s))(Z_y - p_{ii}(s))\right]\\
&\leq \sum_{x \in \partial T^k(s)}\sum_{y \in \partial T^k(s)} \sqrt{\var_i[Z_x]\var_i[Z_y]}\\
&= m^2 p_{ii}(s)(1-p_{ii}(s))\\
&\leq m^2 (1- e^{-q_i s}),
\end{align*}
where on the last line we used that 
the probability of being at state $i$ at
time $s$ is at least the probability of 
never having left state $i$ up to time $s$,
i.e., $e^{-q_i s} \leq p_{ii}(s) \leq 1$ (see e.g.~\cite[Chapter 2]{Liggett:10}).
The result by Chebyshev's inequality.
\end{proof}
\begin{lemma}[$N^{k,s}_\mathcal{A}$ is close to its expectation given $\mathcal{E}^0_\delta $]\label{lem:nkj}
Fix a subset $\mathcal{A} \subseteq \mathcal{S}$.
Let $\delta > 0$. Then, the following bound holds
$$
\P^i\left[
N_\mathcal{A}^{k,s}
<
p_{i\mathcal{A}}(h^*)\,m
- 
[(1- e^{-q_i s})  + 2 \delta] m
\,\middle|\,
\mathcal{E}^0_\delta 
\right]
\leq
\exp
\left(
-  \frac{2 \delta^2}{1+\delta} m
\right).
$$
\end{lemma}
\begin{proof}
We proceed in three steps:
\begin{enumerate}
\item {\bf Conditional control of $\widehat{N}_\mathcal{A}$.} 
Condition on $S_i$.
Define the event
$$
\mathcal{E}^1_\delta 
=
\left\{
\widehat{N}_\mathcal{A} 
< \E[\widehat{N}_\mathcal{A}\,|\,S_i] 
-
\delta m
\right\}.
$$
Here
\begin{align}\label{eq:e-hatnj}
\E\left[\widehat{N}_\mathcal{A}\,\middle|\,S_i\right] 
= p_{i\mathcal{A}}(h^* - s)\,  S_i.
\end{align}
By Hoeffding's inequality~\cite{Hoeffding:63},
we then have
\begin{align}
\P
\left[
\mathcal{E}^1_\delta 
\,\middle|\,
S_i
\right] 
&\leq 
\exp
\left(
- 2 \frac{\delta^2 m^2}{S_i} 
\right).\label{eq:error-1}
\end{align}

\item {\bf Approximation of $p_{i\mathcal{A}}(h^*)$.} 
By~\eqref{eq:e-si} and~\eqref{eq:e-hatnj},
the expectation of $\widehat{N}_\mathcal{A}$ is
$$
\E^i\left[\widehat{N}_\mathcal{A}\right]
= p_{i\mathcal{A}}(h^* - s)\,p_{ii}(s)\, m.
$$
To relate it to the expectation of $N_\mathcal{A}^{k,s}$,
we note that
$$
p_{i\mathcal{A}}(h^*)
=
\sum_{k \in \mathcal{S}}
p_{i k}(s)\, p_{k\mathcal{A}} (h^* - s)
=
p_{i i}(s)\, p_{i\mathcal{A}} (h^* - s)
+
\sum_{k \neq i}
p_{i k}(s)\, p_{k\mathcal{A}} (h^* - s),
$$
so
\begin{align}
\left|
p_{i\mathcal{A}}(h^*)
-
p_{i\mathcal{A}}(h^* - s)\,p_{ii}(s)
\right|
&\leq \sum_{k \neq i}
p_{i k}\, (s) p_{k\mathcal{A}} (h^* - s)\nonumber\\
&\leq \sum_{k \neq i}
p_{i k}(s)\nonumber\\
&= 
1- p_{ii}(s)\nonumber\\
&\leq 1- e^{-q_i s}.\label{eq:approx-pijl}
\end{align}

\item {\bf Overall error.}
Under the event $(\mathcal{E}^0_\delta)^c \cap (\mathcal{E}^1_\delta)^c$, we have 
by~\eqref{eq:e-si} and~\eqref{eq:e-hatnj}
that
$$
\widehat{N}_\mathcal{A} \geq  p_{i\mathcal{A}}(h^* - s)\,p_{ii}(s)\,m 
- 2 \delta m,
$$
where we used $p_{ii}(s) \leq 1$.
In turn, by~\eqref{eq:approx-pijl} and ~\eqref{eq:approx-nkj},
\begin{align*}
N_\mathcal{A}^{k,s}
-
p_{i\mathcal{A}}(h^*)\,m
&\geq
N_\mathcal{A}^{k,s}
-
p_{i\mathcal{A}}(h^* - s)\,p_{ii}(s)\,m
- (1-e^{-q_i s}) m\\
&\geq 
N_\mathcal{A}^{k,s}
-
\widehat{N}_\mathcal{A}
- (1- e^{-q_i s}) m - 2 \delta m\\
&\geq 
- (1- e^{-q_i s}) m - 2 \delta m.
\end{align*}
Define the event
$$
\mathcal{E}^2_\delta
=
\left\{
N_\mathcal{A}^{k,s}
<
p_{i\mathcal{A}}(h^*)\,m
- 
[(1- e^{-q_i s})  + 2 \delta] m
\right\}.
$$
Thus, by the above,
\begin{align*}
\P^i
[
\mathcal{E}^2_\delta
]
&\leq
\P^i
[
\mathcal{E}^0_\delta 
\cup 
\mathcal{E}^1_\delta
]\\
&\leq
\P^i
[
\mathcal{E}^0_\delta 
]
+
\P^i
[
\mathcal{E}^1_\delta
\cap (\mathcal{E}^0_\delta)^c 
]\\
&\leq
\P^i
[
\mathcal{E}^0_\delta 
]
+
\P^i
[
\mathcal{E}^1_\delta\,
|\,(\mathcal{E}^0_\delta)^c 
]\\
&\leq 
\frac{1- e^{-q_i s}}{\delta^2}
+
\exp
\left(
- 2 \frac{\delta^2 m^2}{[p_{ii}(s)+\delta]m} 
\right)\\
&\leq 
\frac{1- e^{-q_i s}}{\delta^2}
+
\exp
\left(
-  \frac{2 \delta^2}{1+\delta} m
\right)
\end{align*}
by~\eqref{eq:error-0}
and~\eqref{eq:error-1}.
\end{enumerate}
That concludes the proof.\end{proof}

\subsection{Analysis of root estimator}
\label{sec:bound-estimator}

We now describe our root estimator. 
In fact, we construct a randomized estimator
(which can be made deterministic by choosing for each input the output most likely to be correct.)
We restrict ourselves to a subset of root states that has high probability under $\pi$ and
we estimate the frequencies of events achieving the total variation distance between the leaf distributions given different root states. 
Fix $\epsilon > 0$ and let $\Lambda = \Lambda_\epsilon$.

\paragraph{Root estimator}
Our root estimator 
$G_k^{\Lambda} : \mathcal{S}^{\partial T^k} \to \mathcal{S}$
is defined as follows. Let $N^{k,s}_{\mathcal{A}}$
and $m$
be defined as in the previous subsection.
\begin{itemize}
	\item Define
	$$
	\Delta
	=
	\inf_{i_1 \neq i_2 \in \Lambda}
	\|
	\mathbf{p}^{i_1}(h^*) 
	- 
	\mathbf{p}^{i_2}(h^*)
	\|_\tv.
	$$
	
	\item For every
	distinct pair of states $i_1, i_2 \in \Lambda$, let $\mathcal{A}_{i_1 \to i_2} \subseteq \mathcal{S}$ be an event achieving the
	total variation distance between $\mathbf{p}^{i_1}(h^*)$ and $\mathbf{p}^{i_2}(h^*)$, that is,
	$$
	\|
	\mathbf{p}^{i_1}(h^*) 
	- 
	\mathbf{p}^{i_2}(h^*)
	\|_\tv
	= 
	p_{i_1,\mathcal{A}_{i_1 \to i_2}}(h^*)
	-
	p_{i_2,\mathcal{A}_{i_1 \to i_2}}(h^*)
	> 0,
	$$
	where we also require that $\mathcal{A}_{i_1 \to i_2} = \mathcal{A}_{i_2 \to i_1}^c$.
	
	\item We let $G^{\Lambda}_k(X^k_{\partial T^k})$ be the state 
	$i$ passing the following tests
	\begin{equation}\label{eq:root-estimator}
	\frac{
		N^{k,s}_{\mathcal{A}_{i \to i'}}
	}
	{
		m
	}
	> 
	p_{i \mathcal{A}_{i \to i'}}(h^*)
	- \frac{\Delta}{2},
	\qquad \forall i' \neq i,
	\end{equation}
	if such a state exists; otherwise
	we let $G^{\Lambda}_k(X^k_{\partial T^k})$ be a state chosen uniformly at random in $\Lambda$.
\end{itemize}
Observe that at most one
state can satisfy the condition in~\eqref{eq:root-estimator}. Indeed,
for any $i \neq i'$, if
$$
\frac{
N^{k,s}_{\mathcal{A}_{i \to i'}}
}
{
m
}
> 
p_{i \mathcal{A}_{i \to i'}}(h^*)
- \frac{\Delta}{2}
$$
then 
$$
\frac{
N^{k,s}_{\mathcal{A}_{i' \to i}}
}
{
m
}
=
1 - \frac{
N^{k,s}_{\mathcal{A}_{i \to i'}}
}
{
m
}
< 
1
-
p_{i \mathcal{A}_{i \to i'}}(h^*)
+ \frac{\Delta}{2}
=
p_{i \mathcal{A}_{i' \to i}}(h^*)
+ \frac{\Delta}{2}
< p_{i' \mathcal{A}_{i' \to i}}(h^*)
- \frac{\Delta}{2},
$$
where we used the definition of $\Delta$ and the fact that $\mathcal{A}_{i \to i'} = \mathcal{A}_{i' \to i}^c$.
Observe also that $G_k^{\Lambda}$ is randomized
as a function of $X^k_{\partial T^k}$ since it depends on the states at the leaves of the extension $\widehat{T}^{k,s}$.

\paragraph{Analysis}
We now prove our main result of this section.

\begin{proof}[Proof of Proposition~\ref{prop:conv-rate}]
Let $F_k
= G^{\Lambda_\epsilon}_k$ be the estimator defined above,
let the events $\mathcal{A}_{i\to i'}$ be as defined above and let $i$ be the state at the root.
By Lemmas~\ref{lem:control-si} and~\ref{lem:nkj},
\begin{align*}
\P^{i}
\left[F_k(X^k_{\partial T^{k}})
\neq i\right]
&=\P^i\left[
\exists i' \neq i,\ 
N_{\mathcal{A}_{i \to i'}}^{k,s}
\leq
p_{i\mathcal{A}_{i \to i'}}(h^*)\,m
- 
\frac{\Delta_{\epsilon}}{2} m
\right]\\
&\leq
\P^i\left[
\exists i' \neq i,\ 
N_{\mathcal{A}_{i \to i'}}^{k,s}
\leq
p_{i\mathcal{A}_{i \to i'}}(h^*)\,m
- 
[(1- e^{-q_i s})  + 2 \delta] m
\right]\\
&\leq
\frac{1- e^{-q_i s}}{\delta^2}
+
n_\epsilon \exp
\left(
-  \frac{2 \delta^2}{1+\delta} m
\right),
\end{align*}
provided $0<2\delta< \frac{\Delta_\epsilon}{2}-(1- e^{-q^*_\epsilon  s})$. 
Take
$
\delta 
=
\frac{\Delta_\epsilon}{8}
$
and 
$s$ small enough that
$1- e^{-q^*_\epsilon s}
\leq 
\Delta_{\epsilon}/4$.
The result follows.
Note finally that,
if $1- e^{-q^*_\epsilon s}
\leq 
\Delta_{\epsilon}/4$
fails, then the bound in Proposition~\ref{prop:conv-rate} is trivially true as the RHS is then larger than $1$.
We leave that condition implicit in the statement.
\end{proof}

\subsection{Uniform chains: minimax error bound}
\label{sec:uniform}

Here we consider chains with unformly bounded rates. 
We give a minimax error bound, that is, a bound uniform in the root state.
We observe in Appendix~\ref{sec:ident} that 
\begin{equation}
\Delta_{Q,h^*}
=
\inf_{i\neq j}\|\mathbf{p}^{i}(h^*)-\mathbf{p}^{j}(h^*)\|_{\tv}\geq \exp{(-h^*\|Q\|)}>0,
\end{equation}
where
$$
\|Q\| 
= \sup_{i} \sum_j |q_{ij}| 
= \sup_{i} 2 q_i < +\infty.
$$
Let
$$
q^* 
= 
\sup_{i \in \mathcal{S}} \,(q_i \lor 1) < +\infty,
$$
and
$$
f_*
=
e^{-q^* h^*}.
$$
Note that
\begin{equation}
\label{eq:fstar}
f^2_*
= e^{- 2 q^* h^*}
\leq \Delta_{Q,h^*}.
\end{equation}
We prove the following proposition, which implies
the second part of Theorem~\ref{thm:2}.
\begin{prop}[Minimax error bound for uniform chains]
\label{prop:conv-rate-unif}
Fix $k \geq 1$. There exist universal constants $C^U_0, C^U_1, C^U_2 > 0$ and an estimator $F^U_k$
such that the following holds.
For all $s > 0$ and all $i$,
\begin{equation*}
\P^{i}
\left[
F^U_k(X^k_{\partial T^{k}})
\neq 
X^k_{\rho}
\right]
<
C^U_0\,  f_*^{-4}\, q^*\, s
+
C^U_2 f_*^{-1} \exp
\left(
-  C^U_1\, f_*^4\, |\partial T^k(s)|
\right).
\end{equation*}
\end{prop}

\paragraph{Root estimator}
We modify the root estimator from Section~\ref{sec:bound-estimator}.
We use the same estimator $G_k^{\Lambda}$,
but we choose a set $\Lambda$ depending
on the leaf states of the extended restriction. 
More precisely,
fix $k \geq 1$ and $s > 0$.
Recall the definitions of $\widehat{T}^{k,s}$ and $N^{k,s}_\mathcal{A}$
from Section~\ref{sec:bound-deviation}.
When $\mathcal{A} = \{j\}$, we write 
$N^{k,s}_{j}$ for $N^{k,s}_{\{j\}}$.
Our modified
estimator is defined as follows.
We let
$$
\widehat{\Lambda}
=
\left\{
j \in \mathcal{S}	
\,:\,
\frac{N^{k,s}_{j}}{|\partial T^k(s)|} \geq \frac{1}{2} f_*
\right\},
$$
and we set
$
F^U_k
= G_k^{\widehat{\Lambda}}.
$

\paragraph{Analysis}
Let $i$ be the state at the root.
Recall the definitions of $S_i$ and $\mathcal{E}^0_\delta$
from Section~\ref{sec:bound-deviation}.
We show first that, conditioned on $\mathcal{E}^0_\delta$,
the set $\widehat{\Lambda}$ is highly likely
to contain $i$, but highly unlikely to contain any state with low enough probability at the leaves. For $\alpha \in [0,1]$, define
$$
\mathcal{J}_{i,\alpha}
=
\left\{
j \in \mathcal{S}
\,:\,
p_{ij}(h^*)
\leq 
\alpha
\right\}.
$$
We write $\mathcal{J}^c_{i,\alpha}$ for $\mathcal{S} \setminus \mathcal{J}_{i,\alpha}$.
Let $m = |\partial T^k(s)|$.
\begin{lemma}[Properties of $\widehat{\Lambda}$]
	\label{lem:hatlambda}
We have
\begin{equation}
\label{eq:hatlambda-i}
\P^i\left[
i \notin
\widehat{\Lambda}
\,\middle|\,
\mathcal{E}^0_\delta
\right]
\leq
\exp\left(
-\frac{f_*^2}{64} m
\right),
\end{equation}
and
\begin{equation}
\label{eq:hatlambda-jf3}
\P^i\left[
\mathcal{J}_{i,f_*/3}
\cap
\widehat{\Lambda}
\neq
\emptyset
\,\middle|\,
\mathcal{E}^0_\delta
\right]
\leq
\left(6 f_*^{-1} + 1\right)
\,
\exp\left(
-\frac{f_*^2}{64} m
\right),
\end{equation}
provided
$1 - e^{-q^* s} \leq f_*/4$
and $\delta \leq f_*/8$.
\end{lemma}
\begin{proof}
For~\eqref{eq:hatlambda-i}, let
$$
\mathcal{A} = \{i\},
$$
and
note that
$$
p_{i\mathcal{A}}(h^*)
\geq 
e^{-q^* h^*}
= f_*.
$$
By Lemma~\ref{lem:nkj}, provided 
$1 - e^{-q^* s} \leq f_*/4$ 
and $\delta \leq f_*/8$, we get
$$
\P^i
\left[
N^{k,s}_i
< \frac{f_*}{2} m
\,\middle|\,
\mathcal{E}^0_\delta
\right]
\leq
\exp\left(
-\frac{f_*^2}{64} m
\right).
$$

For~\eqref{eq:hatlambda-jf3}, consider a partition
$\bigsqcup_{r=1}^R \mathcal{H}_{i,r}$
of $\mathcal{J}_{i,f_*/3}$ into the {\em smallest} number of 
subsets with 
$$
p_{i \mathcal{H}_{i,r}}(h^*) \leq f_*/3.
$$
Observe that $R \leq 6/f_* + 1$. Indeed $\sum_{r=1}^R p_{i \mathcal{H}_{i,r}}(h^*) \leq 1$ and, if
two sets in the partition have $p_{i \mathcal{H}_{i,r}}(h^*) \leq f_*/6$, then they can be combined into one.
By Lemma~\ref{lem:nkj}, provided 
$1 - e^{-q^* s} \leq f_*/4$ 
and $\delta \leq f_*/8$, we get
$$
\P^i
\left[ \exists r,
N^{k,s}_{\mathcal{H}_{i,r}}
\geq \frac{f_*}{2} m
\,\middle|\,
\mathcal{E}^0_\delta
\right]
\leq
R
\,
\exp\left(
-\frac{f_*^2}{64} m
\right).
$$
Noting that $N^{k,s}_{\mathcal{H}_{i,r}} < f_*/2$
implies $N^{k,s}_{j} < f_*/2$ for all $j \in \mathcal{H}_{i,r}$ concludes the proof.
\end{proof}

Recall from Section~\ref{sec:bound-estimator}
the definition of the events $\mathcal{A}_{i\to i'}$.
By the previous lemma, the set $\widehat{\Lambda}$
is likely to contain only elements from $\mathcal{J}_{i,f_*/3}^c$---not necessarily all of them,
but at least the root state $i$.
We show next that under $\mathcal{E}^0_\delta$
the state $i$ is likely to be chosen against
all other states in $\mathcal{J}_{i,f_*/3}^c$ in the tests
performed under $G_k^{\mathcal{J}_{i,f_*/3}^c}$.
\begin{lemma}[Full set of potential tests]
	\label{lem:full-set}
We have
\begin{equation}
\P^i\left[
\exists i' \in \mathcal{J}_{i,f_*/3}^c\setminus\{i\},\ 
N_{\mathcal{A}_{i \to i'}}^{k,s}
\leq
p_{i\mathcal{A}_{i \to i'}}(h^*)\,m
- 
\frac{\Delta_{Q,h^*}}{2} m
\,\middle|\,
\mathcal{E}^0_\delta
\right]
\leq
3f_*^{-1} \exp
\left(
-  \frac{\Delta_{Q,h^*}^2}{64} m
\right),
\end{equation}
provided $1- e^{-q^* s}
\leq 
\Delta_{Q,h^*}/4$
and $
\delta 
\leq
\Delta_{Q,h^*}/8
$.
\end{lemma}
\begin{proof}
We use an argument similar to that in the proof of Proposition~\ref{prop:conv-rate}.Observe first
that
$$
\left|
\mathcal{J}_{i,f_*/3}^c
\right|
\leq \frac{3}{f_*}.
$$
By Lemmas~\ref{lem:control-si} and~\ref{lem:nkj},
\begin{align*}
&\P^i\left[
\exists i' \in \mathcal{J}_{i,f_*/3}^c\setminus\{i\},\ 
N_{\mathcal{A}_{i \to i'}}^{k,s}
\leq
p_{i\mathcal{A}_{i \to i'}}(h^*)\,m
- 
\frac{\Delta_{Q,h^*}}{2} m
\,\middle|\,
\mathcal{E}^0_\delta
\right]\\
&\qquad \leq\P^i\left[
\exists i' \in \mathcal{J}_{i,f_*/3}^c\setminus\{i\},\ 
N_{\mathcal{A}_{i \to i'}}^{k,s}
\leq
p_{i\mathcal{A}_{i \to i'}}(h^*)\,m
- 
[(1- e^{-q_i s})  + 2 \delta] m
\,\middle|\,
\mathcal{E}^0_\delta
\right]\\
&\qquad \leq
\frac{3}{f_*} \exp
\left(
-  \frac{2 \delta^2}{1+\delta} m
\right),
\end{align*}
provided $0<2\delta< \frac{\Delta_{Q,h^*}}{2}-(1- e^{-q^*  s})$. 
Take
$
\delta 
=
\frac{\Delta_{Q,h^*}}{8}
$
and 
$s$ small enough that
$1- e^{-q^* s}
\leq 
\Delta_{Q,h^*}/4$.
The result follows.
\end{proof}

Finally, we prove Proposition~\ref{prop:conv-rate-unif}.

\begin{proof}[Proof of Proposition~\ref{prop:conv-rate-unif}]
Set
$
F_k
= G_k^{\widehat{\Lambda}},
$
and
let $i$ be the root state.
We let $\mathcal{E}^3$ and $\mathcal{E}^4$ be the events
$$
\mathcal{E}^3
= 
\left\{
i \in
\widehat{\Lambda}
\right\}
\cap
\left\{
\mathcal{J}_{i,f_*/3}
\cap
\widehat{\Lambda}
=
\emptyset
\right\},
$$
and
$$
\mathcal{E}^4
= 
\left\{
\forall i' \in \mathcal{J}_{i,f_*/3}^c\setminus\{i\},\ 
N_{\mathcal{A}_{i \to i'}}^{k,s}
>
p_{i\mathcal{A}_{i \to i'}}(h^*)\,m
- 
\frac{\Delta_{Q,h^*}}{2} m
\right\}.
$$
Under $\mathcal{E}^3\cap\mathcal{E}^4$,
it holds that $F_k(X^k_{\partial T^k}) = i$.
Thus, by Lemmas~\ref{lem:control-si},~\ref{lem:hatlambda} and~\ref{lem:full-set},
\begin{align*}
\P^i[F_k(X^k_{\partial T^k}) \neq X^k_\rho]
&\leq 
\P^i[(\mathcal{E}^0_\delta)^c]
+ \P^i\left[(\mathcal{E}^3)^c\,\middle|\,\mathcal{E}^0_\delta\right]
+ \P^i\left[(\mathcal{E}^4)^c\,\middle|\,\mathcal{E}^0_\delta\right]\\
&\leq 
\frac{1- e^{-q_i s}}{\delta^2}
+
\left(6 f_*^{-1} + 2\right)
\,
\exp\left(
-\frac{f_*^2}{64} m
\right)
+
3 f_*^{-1} \exp
\left(
-  \frac{\Delta_{Q,h^*}^2}{64} m
\right)\\
&\leq 
\frac{1- e^{-q^* s}}{\delta^2}
+
11 f_*^{-1}
\exp
\left(
-  \frac{(f_* \land \Delta_{Q,h^*})^2}{64} m
\right),
\end{align*}
provided 
$\delta \leq (f_*\land \Delta_{Q,h^*})/8$
and
$1 - e^{-q^* s} \leq (f_*\land \Delta_{Q,h^*})/4$.
As we did in Proposition~\ref{prop:conv-rate},
the latter condition is implicit.
Using~\eqref{eq:fstar} concludes the proof.
\end{proof}

\section*{Acknowledgments}

We thank Tom Kurtz, Ramon van Handel, and Mykhaylo Shkolnikov for helpful discussions.

\newpage

\bibliographystyle{alpha}
\bibliography{my,thesis}

\newpage

\appendix

\section{Identifiability of initial state on countable state spaces}
\label{sec:ident}

We establish initial-state identifiability in two broad classes of chains.

\paragraph{Uniform chains}
Assume the uniform bound $\sup_{i\in\mathcal{S}} q_{i} < \infty$, where recall that $q_i$ was defined in~\eqref{eq:def-qi}. That condition
implies that $Q$ is a bounded operator on the Banach space $\ell_1$ and that $\mathbf{P}_t=\exp{(tQ)}$ for $t\in[0,\infty)$ (see e.g.~\cite{Ken66}). It then follows that $\mathbf{P}_t$ has an inverse $\mathbf{P}_t^{-1}=\exp{(-tQ)}$ and that
\begin{equation}\label{eq:l1-q}
\inf_{x\neq 0}\frac{\|x\mathbf{P}_t\|_1}{\|x\|_1}\geq \frac{1}{\|\exp{(-tQ)}\|}\geq \exp{(-t\|Q\|)},
\end{equation}
where $\|\cdot\|_{1}$ is the usual $\ell_1$-norm and 
$$
\|Q\| = \sup_{i} \sum_j |q_{ij}| < +\infty,
$$ 
is the operator norm of $Q$. The first inequality in~\eqref{eq:l1-q} follows by putting $y=x\mathbf{P}_t$ and operator $A=\exp{(-tQ)}$ in the inequality $\|y\,A\|_1 \leq \|y\|_1\,\|A\|$. The second inequality follows from 
    $$\| \exp{(-tQ)}\| = \sum_{k\geq 0}\frac{\|(-tQ)^k\|}{k!}\leq \sum_{k\geq 0}\frac{t^k\|Q\|^k}{k!}=\exp{(t\|Q\|)}.$$
Since $\mathbf{p}^{i}(t)=\vec{e}_i\mathbf{P}_t$, we have 
	\begin{equation}\label{ell_1>0}
	\inf_{i\neq j}\|\mathbf{p}^{i}(t)-\mathbf{p}^{j}(t)\|_{\tv}\geq \exp{(-t\|Q\|)}>0,
	\end{equation}
    where we used that the total variation distance is half the $\ell_1$ norm.
This includes the finite state-space case; see~\cite[Lemma 5.1]{GascuelSteel:10} for another proof in that case.

We are unaware of a proof that initial-state identifiability holds more generally in the unbounded case. 
In particular, arguing through the inverse as above may be difficult, as it is related to the longstanding Markov group conjecture. See~\cite{Ken66}. However we argue next that, in the special case of reversible chains, initial-state identifiability does hold in general.

\paragraph{Reversible chains}

Assume now that $(\mathbf{P}_t)_t$ is reversible
(or weakly symmetric) with respect to the
positive measure $\mu$ on $\mathcal{S}$.
By Kendall's representation (see e.g.~\cite[Theorem 1.6.7]{Anderson:91}), for each pair $i, j \in \mathcal{S}$,
there is a finite signed measure $\phi_{ij}$
on $[0,\infty)$ such that
$$
p_{ij}(t)
= \sqrt{\frac{\mu_j}{\mu_i}}
\int_{[0,\infty)}
e^{-tx} d\phi_{ij}(x),
\qquad \forall t\geq 0.
$$
By Jordan decomposition,
$\phi_{ij}$ is the difference of two
finite non-negative measures (see e.g. \cite[Chapter 11]{Royden:88})
so that $p_{ij}(t)$ can be seen as the 
difference of two Laplace transforms
of non-negative measures. The latter
are absolutely convergent, and therefore, analytic on $(0,\infty)$ (see e.g.~\cite[Chapter II]{Widder:41}).
Hence all $p_{ij}(t)$s are analytic.
If two analytic functions agree on a set with a limit point, then they agree everywhere (see e.g.~\cite{Rudin:76}). 
Suppose there exists $t_0>0$ such that $p_{ij}(t_0)=p_{kj}(t_0)$ for all $j$. Then by the Chapman-Kolmogorov equations we have  
	\begin{equation}\label{E:reversible}
	p_{ij}(t)=p_{kj}(t),\qquad \forall t\geq t_0, j\in\mathcal{S}.
	\end{equation}
Then the same holds for all $t > 0$ and, by continuity at $0$, we must have $p_{ii}(0) = p_{ki}(0)$ which implies $i = k$.

\section{An application: the TKF91 process}
\label{sec:tkf}

In this section, we apply Theorem~\ref{thm:1} to ancestral sequence reconstruction in a DNA model accounting for nucleotide insertion and deletion
known as the TKF91 process. 
We first describe the Markovian dynamics. Conforming with the original definition of the model~\cite{thorne1991evolutionary},
we use an ``immortal link'' as a stand-in for the empty sequence.
\begin{definition}[TKF91 sequence evolution model on an edge]\label{Def_TKF91}	
	The  {\bf TKF91 edge process} is a Markov process $\mathcal{I}=(\mathcal{I}_t)_{t\geq 0}$ on the space $\mathcal{S}$ of DNA sequences together with an {\bf immortal link}  ``$\bullet$", that is,
	\begin{equation}\label{S}
	\mathcal{S} := ``\bullet" \otimes \bigcup_{M\geq 0} \{A,T,C,G\}^M,
	\end{equation}
	where the notation above indicates that all sequences begin with the immortal link (and can otherwise be empty).
	We also refer to the positions of a sequence (including nucleotides and the immortal link) as {\bf sites}. 
	Let $(\nu,\,\lambda,\,\mu)\in (0,\infty)^3$ with $\lambda < \mu$ and $(\pi_A,\,\pi_T,\,\pi_C,\,\pi_G)\in [0,\infty)^4$ with $\pi_A +\pi_T + \pi_C + \pi_G = 1$ be given parameters. The continuous-time Markovian dynamic is described as follows: if the current state is the sequence $\vec{x}$, then the following events occur independently:
	\begin{itemize}
		\item (Substitution)$\;$ Each nucleotide (but not the immortal link) is substituted independently at rate $\nu>0$. When a substitution occurs, the corresponding nucleotide is replaced by $A,T,C$ and $G$ with probabilities $\pi_A,\pi_T,\pi_C$ and $\pi_G$ respectively.
		
		\item (Deletion)$\;$ Each nucleotide (but not the immortal link) is removed independently at rate $\mu>0$.
		
		\item (Insertion) $\;$ Each site gives birth to a new nucleotide independently at rate $\lambda>0$. When a birth occurs, a nucleotide is  added immediately to the right of its parent site. The newborn site has nucleotide $A,T,C$ and $G$ with probabilities $\pi_A,\pi_T,\pi_C$ and $\pi_G$ respectively. 
		
	\end{itemize}
	The {\bf length} of a sequence $\vec{x}=(\bullet,x_1,x_2,\cdots,x_M)$ is defined as the number of nucleotides in $\vec{x}$ and is denoted by $|\vec{x}|=M$ (with the immortal link alone corresponding to $M=0$). When $M\geq 1$ we omit the immortal link for simplicity and write $\vec{x}=(x_1,x_2,\cdots,x_M)$.
\end{definition}
\noindent The TKF91 edge process is reversible~\cite{thorne1991evolutionary}. 
Suppose furthermore that 
$$
0 < \lambda < \mu,
$$ 
an assumption we make throughout. Then it has an {\bf stationary distribution} $\Pi$, given by
\begin{equation*}
\Pi(\vec{x})=
\left(1-\frac{\lambda}{\mu}\right) 
\left(\frac{\lambda}{\mu}\right)^M\prod_{i=1}^M\pi_{x_i} 
\end{equation*}
for each $\vec{x}=(x_1,x_2,\cdots,x_M)\in \{A,T,C,G\}^M$ where $M\geq 1$, and $\Pi(``\bullet") = \left(1-\frac{\lambda}{\mu}\right) $. In words, under $\Pi$, the sequence length is geometrically distributed and,
conditioned on the sequence length, all sites are independent with distribution $(\pi_{\sigma})_{\sigma\in \{A,T,C,G\}}$. Hence, from the argument in Section~\ref{sec:ident}, initial-state identifiability holds for the TKF91 edge process.
Theorem~\ref{thm:1} gives:
\begin{theorem}[TKF91 process: consistent root estimation]
Let $\{T^k\}_k$ satisfy assumption (i) and the big bang condition. Let $(\mathbf{P}_t)_t$ be the TKF91 edge process with $\lambda < \mu$ and let $\pi$ be the stationary distribution of the process. Then there exists a sequence of consistent
root estimators.
\end{theorem}
\noindent In a companion paper~\cite{FanRoch:u}, we give an alternative consistent root estimator that is also computationally efficient and provide error bounds that are explicit in the parameters of the model.

\end{document}